\documentclass[12pt,a4paper]{amsart}

\usepackage[utf8]{inputenc}
\usepackage[american]{babel}
\usepackage{csquotes}
\usepackage{mathpazo}
\usepackage{amsmath,amssymb,amsthm}
\usepackage{enumerate}
\usepackage{microtype}
\usepackage[left=2.5cm,right=2.5cm,top=3cm,bottom=3cm]{geometry}
\usepackage[style=alphabetic,sorting=anyt,backend=biber,doi=false,url=false,maxbibnames=10,isbn=false]{biblatex}
\usepackage{color}
\usepackage[hidelinks]{hyperref}
\usepackage{cleveref}

\renewcommand{\AA}{\mathfrak A}
\newcommand{\A}{\mathcal{A}}

\newcommand{\IC}{\mathbb{C}}

\newcommand{\IN}{\mathbb{N}}
\newcommand{\IR}{\mathbb{R}}
\renewcommand{\H}{\mathcal{H}}
\newcommand{\HH}{\mathfrak H}
\newcommand{\J}{\mathcal{J}}
\newcommand{\K}{\mathcal K}
\newcommand{\M}{\mathcal{M}}
\newcommand{\N}{\mathcal{N}}
\newcommand{\U}{\mathcal U}

\newcommand{\E}{\mathcal{E}}
\newcommand{\F}{\mathcal{F}}
\renewcommand{\L}{\mathcal{L}}

\newcommand{\norm}[1]{\lVert#1\rVert}
\newcommand{\abs}[1]{\lvert#1\rvert}
\newcommand{\id}{\mathrm{id}}

\newcommand{\tr}{\operatorname{tr}}

\newcommand{\m}{\mathfrak{m}}
\newcommand{\n}{\mathfrak{n}}
\renewcommand{\epsilon}{\varepsilon}
\renewcommand{\phi}{\varphi}
\renewcommand{\Im}{\mathrm{Im}}
\renewcommand{\Re}{\mathrm{Re}}
\newcommand{\iprod}[2]{\langle\,#1\,,#2\,\rangle}
\newcommand{\IZ}{\mathbb Z}
\newcommand{\op}{\mathrm{op}}

\newtheorem{proposition}{Proposition}[section]
\newtheorem{lemma}[proposition]{Lemma}
\newtheorem{corollary}[proposition]{Corollary}
\newtheorem{theorem}[proposition]{Theorem}

\theoremstyle{remark}
\newtheorem{remark}[proposition]{Remark}
\newtheorem{example}[proposition]{Example}

\theoremstyle{definition}
\newtheorem{definition}[proposition]{Definition}

\title[Differential structure of QMS]{The differential structure of generators of GNS-symmetric quantum Markov semigroups}
\author{Melchior Wirth}
\address{Institute of Science and Technology Austria (ISTA), Am Campus 1, 3400 Kloster\-neuburg, Austria}
\email{melchior.wirth@ist.ac.at}

\begin{document}
\begin{abstract}
We show that the generator of a GNS-symmetric quantum Markov semigroup can be written as the square of a derivation. This generalizes a result of Cipriani and Sauvageot for tracially symmetric semigroups. Compared to the tracially symmetric case, the derivations in the general case satisfy a twisted product rule, reflecting the non-triviality of their modular group. This twist is captured by the new concept of Tomita bimodules we introduce. If the quantum Markov semigroup satisfies a certain additional regularity condition, the associated Tomita bimodule can be realized inside the $L^2$ space of a bigger von Neumann algebra, whose construction is an operator-valued version of free Araki-Woods factors.
\end{abstract}

\maketitle


\section{Introduction}

Quantum Markov semigroups, originally introduced in the study of the time evolution of open quantum systems \cite{Ali76,GKS76,Lin76}, have long since permeated the boundaries of mathematical physics and become an object of interest across various disciplines of (noncommutative) mathematics such as noncommutative harmonic analysis \cite{JX07,JMP14}, noncommutative probability, noncommutative geometry \cite{Sau96} and the structure theory of von Neumann algebras \cite{CS15,CS17}.

One of the central questions from the very beginning was to characterize the generators of quantum Markov semigroups. For quantum Markov semigroups on matrix algebras, the generators are characterized by the theorems of Lindblad \cite{Lin76} and Gorini--Kossakowski--Sudarshan \cite{GKS76}, while the generators of uniformly continuous quantum Markov semigroups on arbitrary von Neumann algebras have been described by Christensen--Evans \cite{CE79}. While quantum Markov semigroups on matrix algebras are of interest in quantum information theory, many examples of classical and quantum Markov semigroups act on infinite-dimensional algebras and are not uniformly continuous, that is, they have unbounded generators.

To facilitate the study of quantum Markov semigroups with unbounded operators, one often focuses on semigroups that satisfy a certain symmetry condition so that one can apply techniques for self-adjoint operators on Hilbert spaces. The simplest such symmetry condition is tracial symmetry, where the reference state is a trace. The study of tracially symmetric quantum Markov semigroups and the associated quadratic forms, completely Dirichlet forms, was initiated by Alberverio--H{\o}egh-Krohn \cite{AHK77} and further pursued by Davies--Lindsay \cite{DL92,DL93}.

In this setting, Cipriani and Sauvageot \cite{CS03} proved that the $L_2$-generator $\L^{(2)}$ can be written as $\L^{(2)}=\delta^\ast\bar\delta$, where $\delta$ is a derivation with values in a Hilbert bimodule. This result spawned a lot of applications across various disciplines, from analysis on fractals \cite{HT13} and metric graphs \cite{BK19} over noncommutative geometry \cite{CS03a}, noncommutative probability \cite{Dab10,JZ15}, quantum optimal transport \cite{Wir20,WZ21a} to the structure theory of von Neumann algebras and in particular Popa's deformation and rigidity theory \cite{Pet09,DI16,Cas21,CIW21}.

Despite this resounding success, the condition of tracial symmetry is too strong in various situations. For example, type III von Neumann algebras do not even admit a faithful normal trace. Quantum Markov semigroups on type III von Neumann algebras occur naturally for example in free probability and the study of compact quantum groups (see Subsections \ref{subsec:Araki-Woods}, \ref{subsec:CQG} for more details). But even on type I and type II von Neumann algebras there are many quantum Markov semigroups for which the natural reference state is not a trace, for example in mathematical physics when modelling of open quantum systems at finite temperature (see below). For this reason the problem of extending Cipriani and Sauvageot's result to a larger class of quantum Markov semigroups has been raised (see \cite{SV19,Cas21} for example).

In this article we prove a representation of the generator as square of a derivation in the spirit of Cipriani and Sauvageot for the class of GNS-symmetric quantum Markov semigroups on arbitrary von Neumann algebras.

The GNS symmetry (or detailed balance) condition for quantum Markov semigroups also comes with a clear motivation from physics -- it describes the time evolution of open quantum systems coupled to an environment in thermal equilibrium \cite{Ali76}. Tracial symmetry is a special case of GNS symmetry, which corresponds to an equilibrium at infinite temperature in this picture. But GNS-symmetric quantum Markov semigroups also occur naturally across other fields of mathematics, for example the free Ornstein--Uhlenbeck semigroup of free Araki--Woods factors or quantum Markov semigroups associated with Lévy processes on compact quantum groups that are not of Kac type (again, see Subsection \ref{subsec:Araki-Woods}, \ref{subsec:CQG} for details).

On matrix algebras, the structure of the generators of GNS-symmetric quantum Markov semigroups is well-understood \cite{Ali76} and can be phrased in terms of derivations \cite{CM17,CM20}. This has given rise to applications to $L_p$-regularity of quantum Dirichlet forms \cite{Bar17}, quantum optimal transport \cite{CM17,MM17,CM20} and logarithmic Sobolev inequalities \cite{GR21}. Especially the first two results also offer interesting possible applications of the results of this article. Recently, also the generators of uniformly continuous GNS-symmetric quantum Markov semigroups have been characterized by the author \cite{Wir22}.

These results already hint at a certain twist needed to represent the generators of GNS-symmetric quantum Markov semigroups as squares of derivations. In fact, it has been shown recently by Vernooij \cite{Ver22} that there exist GNS-symmetric quantum Markov semigroups on matrix algebras that cannot be written as $\L^{(2)}=\delta^\ast\delta$ for a derivation $\delta$ with values in a Hilbert bimodule. Hence, to extend Cipriani and Sauvageot's result to GNS-symmetric quantum Markov semigroups, new notions are required.

The characterizing property of traces on von Neumann algebras is that they have trivial modular group. To handle the case when the reference weight is not a trace, the interaction with the modular group gives rise to new features. In particular, a GNS-symmetric quantum Markov semigroups commutes with the modular group, a fact which we systematically exploit. It implies that the bimodule we construct comes with certain extra structure resembling the modular group and conjugation on a Tomita algebra, which we capture in our definition of Tomita bimodules. While the analog of the conjugation operator is also present in the tracially symmetric case, the analog of the modular group is a new phenomenon in the non-tracial case.

Another way to phrase that the modular group of a trace is trivial is to say that the left and right Hilbert algebra associated with it coincide. The twist that appears only in the non-tracial case is that the left action on the bimodule we construct is a $\ast$-homomorphism with respect to the involution from the left Hilbert algebra, while the right action is a $\ast$-anti-homomorphism with respect to the involution from the right Hilbert algebra. 

Let us now describe the contents of this article in more detail. In \Cref{sec:prelim} we review some basic von Neumann algebra theory, in particular Tomita--Takesaki theory and modules over von Neumann algebras. In Section \Cref{sec:QMS} we recall some known facts on quantum Markov semigroups and quantum Dirichlet forms and characterize conservativeness in terms of the Dirichlet form (\Cref{prop:char_Dir_form_conservative}).

In \Cref{sec:Tomita_bimodule} we introduce the new notion of Tomita bimodules (\Cref{def:Tomita_bimodule}), which serve as the codomains of the derivations we associate with GNS-symmetric quantum Markov semigroups later. Using a Fock space construction, we show that a Tomita bimodule $\H$ over the Tomita algebra $\AA$ can be embedded into the Hilbert completion of a Tomita algebra $\mathfrak B\supset \AA$, provided the actions of $\AA$ are normal in a suitable sense (Subsection \ref{subsec:Fock_construction}).

In \Cref{sec:sym_deriv} we consider derivations on Tomita algebras with values in Tomita bimodules that respect the modular automorphism group and conjugation operator (\Cref{def:sym_deriv}). We show that if such a derivation $\delta$ is inner or more generally can be approximated by inner derivations in a suitable sense, then $\delta^\ast\bar\delta$ generates a GNS-symmetric quantum Markov semigroup (\Cref{prop:from_bdd_deriv_to_QMS}, \Cref{thm:from_deriv_to_QMS}).

In \Cref{sec:from_QMS_to_deriv} we study the reverse construction going from a quantum Markov semigroup with $L_2$-generator $\L^{(2)}$ to a derivation such that $\delta^\ast\bar\delta=\L^{(2)}$. We first show how the results from \cite{Wir22}, upgraded to cover the case of a reference weight instead of a state, can be used to prove the existence in the case of a bounded generator (\Cref{prop:deriv_bdd}).

To deal with unbounded generators, one problem is that of an appropriate domain for the associated derivation. For the product rule to make sense, $\delta$ needs to be defined on an algebra. In the tracial setting, a suitable algebra was identified by Davies and Lindsay \cite{DL92}. Here we show that for a GNS-symmetric quantum Markov semigroup with associated quantum Dirichlet form $\E$, the set
\begin{equation*}
\AA_\E=\{a\in \AA_\phi\mid U_z a\in D(\E)\text{ for all }z\in\IC\}
\end{equation*}
is a Tomita algebra and a core for $\E$ (\Cref{thm:Tomita_algebra_QMS}). Here $\AA_\phi$ denotes the maximal Tomita algebra associated with the weight $\phi$.

We then prove the main result of this article, namely the existence a Tomita bimodule $\H$ over $\AA_\E$ and a derivation $\delta\colon\AA_\E\to\H$ such that the $L_2$-generator satisfies $\L^{(2)}=\delta^\ast\bar\delta$ (\Cref{thm:from_QMS_to_deriv}). Further we show that under natural assumptions, the Tomita bimodule $\H$ and the derivation $\delta$ are unique (\Cref{thm:deriv_unique}).

A desirable property of Tomita bimodules is that the left and right action are normal, which is used for example in our Fock space construction. Even if the underlying von Neumann algebra is commutative, it is known that this is not necessarily the case for the Tomita bimodule associated with a quantum Markov semigroup. In \Cref{sec:carre_du_champ} we give a characterization of GNS-symmetric quantum Markov semigroups that give rise to normal Tomita bimodules that runs parallel to the tracially symmetric case (\Cref{thm:char_normal_bimodule}). In this case, the quantum Markov semigroups induces a structure analog to the derivation triples from \cite{JRS} in the tracially symmetric case (\Cref{cor:deriv_triple}).

Finally, in \Cref{sec:examples} we describe the derivation associated with a GNS-symmetric quantum Markov semigroups in concrete examples. First we show how tracially symmetric quantum Markov semigroups and GNS-symmetric quantum Markov semigroups on finite-dimensional von Neumann algebras fit into the picture (Subsections \ref{subsec:trace_sym_QMS}, \ref{subsec:fin_dim_QMS}). Then we discuss some genuinely new examples with depolarizing semigroups, the Ornstein--Uhlenbeck semigroup on free Araki--Woods factors and translation-invariant quantum Markov semigroups on compact quantum groups not of Kac type (Subsections \ref{subsec:depolarizing_QMS}--\ref{subsec:CQG}).

\subsection*{Acknowledgments} The author wants to thank Martijn Caspers, Matthijs Vernooij and Jan Maas for fruitful discussions on the topic. He is particularly grateful to Matthijs Vernooij and Haonan Zhang for a number of helpful comments on a preliminary version of this article. This research was partially funded by the Austrian Science Fund (FWF) under the Esprit Programme [ESP 156] and the European Research Council (ERC) under the European Union’s Horizon 2020 research and innovation programme (grant agreement No 716117). For the purpose of Open Access, the author has applied a CC BY public copyright licence to any Author Accepted Manuscript (AAM) version arising from this submission.


\section{Preliminaries}\label{sec:prelim}

In this section we briefly recall some definitions and results on Neumann algebras, in particular on weights, Hilbert and Tomita algebras and correspondences, and we fix the notation for the rest of the article. For a more in-depth account of the material presented here, which is mostly standard, see \cite{Tak03}. A detailed account of $C^\ast$- and $W^\ast$-modules can be found in \cite{Ske01}.

To start with, let us fix some general notation. All vector spaces are over the complex numbers, unless explicitly stated otherwise, and (semi-) inner products are linear in the second argument. If $V$ and $W$ are vector spaces, then $V\odot W$ denotes their algebraic tensor product, while $\otimes$ is used for the tensor product of Hilbert spaces. Some other tensor products will be discussed below.

\subsection{Weights and semi-cyclic representations}

Let $\M$ a von Neumann algebra. A \emph{weight} on $\M$ is a map $\phi\colon \M_+\to[0,\infty]$ such that $\phi(\lambda x)=\lambda\phi(x)$ and $\phi(x+y)=\phi(x)+\phi(y)$ for all $x,y\in\M_+$ and $\lambda\geq 0$. Here we use the convention $0\cdot\infty=0$.

For a weight $\phi$ we define
\begin{align*}
\mathfrak p_\phi&=\{x\in\M_+\mid \phi(x)<\infty\},\\
\n_\phi&=\{x\in \M\mid x^\ast x\in \mathfrak p_\phi\},\\
\m_\phi&=\operatorname{lin}\{x^\ast y\mid x,y\in\n_\phi\}.
\end{align*}
The weight $\phi$ is called \emph{normal} if $\sup_j \phi(x_j)=\phi(\sup_j x_j)$ for every bounded increasing net $(x_j)$ in $\M_+$, \emph{semi-finite} if $\mathfrak p_\phi$ generates $\M$ as a von Neumann algebra, and \emph{faithful} if $\phi(x^\ast x)=0$ implies $x=0$.

Every element of $\m_\phi$ is a linear combination of four elements of $\mathfrak p_\phi$, and $\phi$ can be linearly extended to $\m_\phi$. This extension will still be denoted by $\phi$.

A semi-cyclic representation of $\M$ is a triple $(\pi,H,\Lambda)$ consisting of a normal representation of $\M$ on $H$ and a $\sigma$-strong$^\ast$ closed linear map $\Lambda$ from a dense left ideal $\n$ of $\M$ into $H$ with dense range such that
\begin{align*}
\pi(x)\Lambda(y)&=\Lambda(xy)
\end{align*}
for all $x\in \M$ and $y\in\n$.

There does not seem to be an established term for $\n$ in this situation. We call it the \emph{definition ideal} of the semi-cyclic representation $(\pi,H,\Lambda)$.

Given a normal semi-finite weight $\phi$ on $\M$, the associated semi-cyclic representation $(\pi_\phi,L_2(\M,\phi),\Lambda_\phi)$ is defined as follows: $L_2(\M,\phi)$ is the Hilbert space obtained from $\n_\phi$ after separation and completion with respect to the inner product 
\begin{equation*}
\langle\cdot,\cdot\rangle_\phi\colon\n_\phi\times\n_\phi\to\IC,\,(x,y)\mapsto\phi(x^\ast y),
\end{equation*}
$\Lambda_\phi\colon \n_\phi\to L_2(\M,\phi)$ is the quotient map and $\pi_\phi$ is given by $\pi_\phi(x)\Lambda_\phi(y)=\Lambda_\phi(xy)$. We also use the alternative notation $x\phi^{1/2}$ for $\Lambda_\phi(x)$.

This semi-cyclic representation is essentially uniquely determined by $\phi$ in the following sense: If $(\pi,H,\Lambda)$ is another semi-cyclic representation of $\phi$ with definition ideal $\n_\phi$ and $\langle \Lambda(x),\Lambda(y)\rangle=\phi(x^\ast y)$ for all $x,y\in\n_\phi$, then there exists a unitary $U\colon L_2(\M,\phi)\to H$ such that $U\Lambda_\phi=\Lambda$ and $U\pi_\phi(x)U^\ast=\pi(x)$ for all $x\in\M$.

\subsection{Hilbert algebras and Tomita algebras}

An algebra $\AA$ with involution $^\sharp$ (resp. $^\flat$) and inner product $\iprod{\cdot}{\cdot}$ is called \emph{left (resp. right) Hilbert algebra} if
\begin{itemize}
\item for every $a\in\AA$ the map $\pi_l(a)\colon\AA\to \AA$, $b\mapsto ab$ (resp. $b\mapsto ba$) is bounded with respect to $\iprod{\cdot}{\cdot}$,
\item $\langle ab,c\rangle=\langle b,a^\sharp c\rangle$ (resp. $\langle ab,c\rangle=\langle b,ca^\flat\rangle$) for all $a,b,c\in\AA$,
\item the involution $^\sharp$ (resp. $^\flat$) is closable,
\item the linear span of all products $ab$ with $a,b\in\AA$ is dense in $\AA$ with respect to $\iprod{\cdot}{\cdot}$.
\end{itemize}
Let $\M$ be a von Neumann algebra and $\phi$ a normal semi-finite faithful weight on $\M$. The prototypical example of a left Hilbert algebra is $\AA=\Lambda_\phi(\n_\phi\cap\n_\phi^\ast)$ with the product $\Lambda_\phi(x)\Lambda_\phi(y)=\Lambda_\phi(xy)$, the involution $\Lambda_\phi(x)^\sharp=\Lambda_\phi(x^\ast)$ and the inner product inherited from $L_2(\M,\phi)$, that is, $\iprod{\Lambda_\phi(x)}{\Lambda_\phi(y)}=\phi(x^\ast y)$. In this case, $\pi_l(\AA)^{\prime\prime}=\pi_\phi(\M)$.

Conversely, every left Hilbert algebra $\AA$ gives rise to a von Neumann algebra $\pi_l(\AA)^{\prime\prime}$ acting on the completion of $\AA$ and a weight 
\begin{equation*}
\phi\colon \pi_l(\AA)^{\prime\prime}_+\to [0,\infty],\,\phi(x)=\begin{cases}\norm{\xi}^2&\text{if }x^{1/2}=\pi_l(\xi),\\\infty&\text{otherwise}.\end{cases}
\end{equation*}
If $\AA$ is a full left Hilbert algebra \cite[Definition VI.1.16]{Tak03}, then $\phi$ is a normal semi-finite faithful weight on $\pi_l(\AA)^{\prime\prime}$.

Let $\HH$ be the completion of the Hilbert algebra $\AA$. Since the involution $^\sharp$ on $\AA$ is closable, its closure $S$ on $\HH$ exists and has a polar decomposition $S=J\Delta^{1/2}$. The operator $\Delta$ is a non-singular positive self-adjoint operator, called the \emph{modular operator}, and $J$ is an anti-unitary involution, called the \emph{modular conjugation}. If $\AA$ is the left Hilbert algebra associated with a weight $\phi$, we write $\Delta_\phi$ and $J_\phi$ for the associated modular operator and modular conjugation.

We write $\Lambda_\phi^\prime\colon \n_\phi^\ast\to L_2(\M,\phi)$ for the map $x\mapsto J_\phi\Lambda_\phi(x^\ast)$. We also use the notation $\phi^{1/2} x$ instead of $\Lambda_\phi^{\prime}(x)$.

If $\AA$ is full, the modular conjugation $J$ gives rise to the positive self-dual cone $P=\overline{\{\pi_l(a)Ja\mid a\in\AA\}}$ and $\pi_l(\AA)^{\prime\prime}$ is in standard form \cite[Definition IX.1.13]{Tak03}.

The modular operator $\Delta$ gives rise to a pointwise $\sigma$-weakly continuous group of automorphisms $x\mapsto \Delta^{it}x\Delta^{-it}$ on $\pi_l(\AA)^{\prime\prime}$. If $\phi$ is a normal semi-finite faithful weight on $\M$, the group $\sigma^\phi$ given by $\sigma^\phi_t(x)=\pi_\phi^{-1}(\Delta_\phi^{it}\pi_\phi(x)\Delta_\phi^{-it})$ is called the \emph{modular group} associated with $\phi$.

If $(\alpha_t)_{t\in\IR}$ is a pointwise weak$^\ast$ continuous group of $\ast$-automorphisms on $\M$, then an element $x\in \M$ is called \emph{entire analytic} if the map $t\mapsto \alpha_t(x)$ has an extension $z\mapsto \alpha_z(x)$ to the complex plane such that $z\mapsto\omega(\alpha_z(x))$ is analytic for every $\omega\in\M_\ast$. The entire analytic elements form a weak$^\ast$ dense $\ast$-subalgebra of $\M$.

A \emph{Tomita algebra} is a left Hilbert algebra $\AA$ endowed with a complex one-parameter group $(U_z)_{z\in\IC}$ of algebra automorphism such that
\begin{itemize}
\item $z\mapsto \langle a,U_z b\rangle$ is analytic for all $a,b\in\AA$,
\item $(U_z a)^\sharp=U_{\bar z}(a^\sharp)$ for all $a\in\AA$, $z\in\IC$,
\item $\langle U_z a,b\rangle=\langle a,U_{-\bar z}b\rangle$ for all $a,b\in\AA$, $z\in \IC$,
\item $\langle a^\sharp,b^\sharp\rangle=\langle U_{-i}b,a\rangle$ for all $a,b\in \AA$.
\end{itemize}
Note that every Tomita algebra becomes a right Hilbert algebra when endowed with the involution
\begin{equation*}
\AA\to\AA,a\mapsto a^\flat=U_{-i}(a^\sharp).
\end{equation*}
For a full left Hilbert algebra $\AA$ let
\begin{equation*}
\AA_0=\left\{\xi\in\bigcap_{n\in\IZ}D(\Delta^n)\,\bigg\vert\, \Delta^n\xi\in\AA\text{ for all }n\in\IZ\right\}.
\end{equation*}
For every $\xi\in\AA_0$ the map $t\mapsto \Delta^{it}\xi$ has an entire analytic extension $z\mapsto U_z\xi$ with $U_z\xi\in \AA_0$ for all $z\in\IC$. This makes $\AA_0$ into a Tomita algebra such that $\pi_l(\AA_0)^{\prime\prime}=\pi_l(\AA)^{\prime\prime}$. If $\AA$ is the left Hilbert algebra associated with a weight $\phi$, we write $\AA_\phi$ for $\AA_0$ and call it the \emph{Tomita algebra associated with $\phi$}.

\subsection{Correspondences and relative tensor product}

Let $\M$, $\N$ be von Neumann algebras. An $\M$-$\N$-correspondence is a Hilbert space $\H$ endowed with commuting normal representations of $\M$ and $\N^\op$. We write $x\xi$ and $\xi y$ for the left and right action, respectively. If $\phi$ is a normal semi-finite weight on $\M$, then $L_2(\M,\phi)$ is an $\M$-$\M$-correspondence with the usual left action and $\xi\cdot y=J_\phi y^\ast J_\phi\xi$.

If $\psi$ is a normal semi-finite faithful weight on $\N$, a vector $\xi\in \H$ is called \emph{left-bounded} (with respect to $\psi$) if there exists a constant $C>0$ such that $\norm{\xi y}\leq C\norm{\psi^{1/2}y}$ for all $y\in \N$. The notation $\psi^{1/2}y$ was explained in the previous subsection. We write $L_\infty(\H_\M,\psi)$ for the set of all left-bounded vectors in $\H$. If $\xi\in L_\infty(\H_\M,\psi)$, then $\phi^{1/2}y\mapsto \xi y$ extends to a bounded linear operator from $L_2(\N,\psi)$ to $\H$, which we denote by $L_\psi(\xi)$.

There is also the dual notion of right-bounded vectors. A vector $\xi\in \H$ is called \emph{right-bounded} (with respect to $\phi$) if there exists a constant $C>0$ such that $\norm{x\xi}\leq C\norm{x\phi^{1/2}}$ for all $x\in \M$. We write $L_\infty(_\M\H,\phi)$ for the set of all right-bounded vectors in $\H$. The bounded operator from $L_2(\M,\phi)$ to $\H$ that extends $x\phi^{1/2}\mapsto x\xi$ is denoted by $R_\phi(\xi)$.

A vector in $\H$ is called \emph{bounded} if it is both left- and right-bounded. We write $L_\infty(_\M\H_\N,(\phi,\psi))$ for the set of all bounded vectors or simply $L_\infty(_\M\H_\M,\phi)$ if $\M=\N$ and $\phi=\psi$.

If $\H$, $\K$ are $\M$-$\N$-correspondences, we write $\L(\H_\N,\K_\N)$ for the set of all bounded right module maps from $\H$ to $\K$. Similarly, $\L(_\M\H,_\M\K)$ denotes the set of all bounded left module maps from $\H$ to $\K$. Clearly $L_\psi(\xi)\in \L(L_2(\N,\psi)_\N,\H_\N)$ for every left-bounded vector $\xi\in \H$ and $R_\psi(\xi)\in \L(_\M L_2(\M,\phi),_\M\H)$ for every right-bounded vector $\xi\in \H$. Since $\M$ and $\N$ are in standard form on $L_2(\M,\phi)$ and $L_2(\N,\psi)$, respectively, we have 
\begin{align*}
\L(L_2(\N,\psi)_\N,L_2(\N,\psi)_\N)&=\N\\
\L(_\M L_2(\M,\phi),_\M L_2(\M,\phi))&=\M^\prime.
\end{align*}
Thus $L_\psi(\xi)^\ast L_\psi(\xi)\in \N$ for every left-bounded vector $\xi\in \H$ and $R_\phi(\xi)^\ast R_\phi(\xi)\in \M^\prime$ for every right-bounded vector $\xi\in \H$. We also write $(\xi|\eta)$ for $L_\phi(\xi)^\ast L_\phi(\eta)$ to stress the connection to Hilbert bimodules.

If $\M$, $\N$ and $\mathcal R$ are von Neumann algebras, $\H$ is an $\M$-$\N$-correspondence, $\K$ is an $\N$-$\mathcal R$-correspondence and $\psi$ is a normal semi-finite faithful weight on $\N$, the \emph{relative tensor product} $\H\otimes_\psi\K$ is the Hilbert space obtained from the algebraic tensor product $L_\infty(\H_\N,\psi)\odot \K$ after separation and completion with respect to the semi-inner product
\begin{equation*}
\langle \xi_1\otimes \eta_1,\xi_2\otimes \eta_2\rangle=\langle \eta_1,L_\psi(\xi_1)^\ast L_\psi(\xi_2)\cdot\eta_2\rangle.
\end{equation*}
This expression makes sense since $L_\psi(\xi_1)^\ast L_\psi(\xi_2)\in \N$. The image of $\xi\otimes\eta$ in $\H\otimes_\psi\K$ is denoted by $\xi\otimes_\psi \eta$.

On $L_\infty(\H_\N,\psi)\odot L_\infty(_\N \K,\psi)$ one has
\begin{equation*}
\langle \eta_1,L_\psi(\xi_1)^\ast L_\psi(\xi_2)\eta_2\rangle=\langle \xi_1,\xi_2\cdot J_\psi R_\psi(\eta_2)^\ast R_\psi(\eta_1)J_\psi\rangle.
\end{equation*}
Hence $\H\otimes_\psi\K$ can equivalently be defined as the Hilbert space obtained from $\H\odot L_\infty(_\N\H,\psi)$ after separation and completion with respect to the semi-inner product
\begin{equation*}
\langle \xi_1\otimes \eta_1,\xi_2\otimes \eta_2\rangle=\langle \xi_1,\xi_2\cdot J_\psi R_\psi(\eta_2)^\ast R_\psi(\eta_1)J_\psi\rangle.
\end{equation*}

\subsection{\texorpdfstring{$C^\ast$-modules and von Neumann modules}{C*-modules and von Neumann modules}}

Let $A$ be a unital $C^\ast$-algebra. A pre-$C^\ast$-module over $A$ is a right $A$-module $F$ with a sesquilinear map $(\cdot|\cdot)\colon F\times F\to A$ such that
\begin{itemize}
\item $(\xi|\eta)x=(\xi|\eta x)$ for all $\xi,\eta\in F$, $x\in A$,
\item $(\xi|\xi)\geq 0$ for all $\xi\in F$,
\item $(\xi|\xi)=0$ if and only if $\xi=0$.
\end{itemize}
A $C^\ast$-module is a pre-$C^\ast$-module that is complete in the norm $\norm{(\cdot|\cdot)}^{1/2}$. 

A bounded linear operator $T$ on a $C^\ast$-module $F$ is called \emph{adjointable} if there exists a bounded linear operator $T^\ast$ on $F$ such that
\begin{equation*}
(T\xi|\eta)=(\xi|T^\ast\eta)
\end{equation*}
for all $\xi,\eta\in F$. Note that all adjointable operators are right module maps, that is, $T(\xi a)=(T\xi)a$ for all $a\in A$, $\xi\in F$.

Let $A$, $B$ and $C$ be unital $C^\ast$-algebras. A \emph{$C^\ast$ $A$-$B$-module} is a $C^\ast$-module over $B$ together with an action of $A$ by adjointable operators. In particular, a $C^\ast$ $\IC$-$A$-bimodule is the same as a $C^\ast$ $A$-module and a $C^\ast$ $A$-$\IC$-bimodule the same as a representation of $A$ on a Hilbert space. In the case $A=B$ we simply speak of $C^\ast$ $A$-bimodules.

The \emph{tensor product} $F\bar\odot G$ of a $C^\ast$ $A$-$B$-module $F$ and a $C^\ast$ $B$-$C$-module is the $C^\ast$ $A$-$C$-module obtained from the algebraic tensor product $F\odot G$ after separation and completion with respect to the $C$-valued inner product given by
\begin{equation*}
(\xi\otimes\eta|\xi'\otimes\eta')=(\eta|(\xi|\xi')\eta')
\end{equation*}
and the actions given by
\begin{equation*}
a(\xi\otimes\eta)=a\xi\otimes \eta,\,(\xi\otimes\eta)c=\xi\otimes\eta c.
\end{equation*}
If $A$ is a $C^\ast$-algebra of bounded operators on the Hilbert space $H$ and $F$ is a $C^\ast$ $A$-module, we can embed $F$ into $B(H,F\bar\odot H)$ by the action 
\begin{equation*}
H\to F\bar\odot H,\,\zeta\mapsto \xi\otimes\zeta
\end{equation*}
for $\xi\in F$. If we refer to the strong topology on a $C^\ast$-module in the following, we always mean the strong topology in this embedding.

Let $\M$ be a von Neumann algebra on $H$. A \emph{von Neumann $\M$-module} is a $C^\ast$ $\M$-module $F$ that is strongly closed in $B(H,F\bar\odot H)$. Several equivalent definitions of von Neumann modules have been given in \cite[Proposition 2.9]{Sch02}. In particular, a $C^\ast$-module over a von Neumann algebra is a von Neumann module if and only if it is isometrically isomorphic to a dual space and the right action is weak$^\ast$ continuous. The adjointable operators on a von Neumann $\M$-module form a von Neumann algebra $\L_\M(F)$.

If $\N$ is another von Neumann algebra on $K$, then a $C^\ast$ $\M$-$\N$-module is a \emph{von Neumann $\M$-$\N$-module} if it is a von Neumann $\N$-module and the left action of $\M$ on $F\bar\odot K$ is normal.

\section{Quantum Markov semigroups and quantum Dirichlet forms}\label{sec:QMS}

In this section we review some basic material on quantum Markov semigroups and the GNS symmetry condition, before we give a new characterization of conservativeness in terms of the associated quadratic forms on $L_2$ (\Cref{prop:char_Dir_form_conservative}) and discuss some density properties.

Let $\M$ be a von Neumann algebra. A \emph{quantum dynamical semigroup} on $\M$ is a pointwise $\sigma$-weakly continuous semigroup $(P_t)$ of normal contractive completely positive operators on $\M$. If $P_t$ is unital for every $t\geq 0$, we call $(P_t)$ a \emph{quantum Markov semigroup}.

Note that these definitions are not universal -- some authors use the term ``quantum Markov semigroup'' for the objects we call quantum dynamical semigroups and ``conservative quantum Markov semigroup'' for our quantum Markov semigroups. The focus of this article is exclusively on quantum Markov semigroups (in our sense), we only introduce quantum dynamical semigroups to relate some results from the literature that are formulated in the non-unital case.

The continuity requirement in the definition of quantum dynamical semigroups implies pointwise continuity for any of the ``reasonable'' locally convex operator topologies weaker than the norm topology. This is well-known among experts, but since we did not find a reference, we give the simple proof here.

\begin{lemma}\label{lem:QMS_strong_conv}
Let $\M$ be a von Neumann algebra. If $(P_t)$ is a quantum dynamical semigroup, then $P_t (x)\to x$ in the $\sigma$-strong$^\ast$ topology as $t\searrow 0$ for every $x\in \M$.
\end{lemma}
\begin{proof}
Let $x\in \M$ and $\omega\in \M_\ast$. By the Kadison--Schwarz inequality,
\begin{equation*}
\omega((x-P_t(x))^\ast(x-P_t(x)))\leq \omega(x^\ast x)+\omega(P_t(x^\ast x))-\omega(x^\ast P_t(x))-\omega(P_t(x)^\ast x)\to 0.
\end{equation*}
Thus $P_t(x)\to x$ in the $\sigma$-strong topology. Since $P_t$ is a self-adjoint map, the same argument applies to $(P_t(x^\ast))$. Therefore $P_t(x)\to x$ in the $\sigma$-strong$^\ast$ topology.
\end{proof}

Now we turn to quantum Markov semigroups that satisfy a suitable symmetry condition which makes them amenable to self-adjoint operator techniques on Hilbert spaces. On general von Neumann algebras, the study of these semigroups in the non-tracial setting goes back to Goldstein and Lindsay \cite{GL95,GL99} and Cipriani \cite{Cip97}.

Now let $\phi$ be a normal semi-finite faithful weight on $\M$. We call a quantum dynamical semigroup $(P_t)$ on $\M$ \emph{GNS-symmetric with respect to $\phi$} or simply $\phi$-symmetric if $\phi\circ P_t\leq \phi$ for all $t\geq 0$ and
\begin{equation}\label{eq:GNS-sym}
\phi(P_t(x)^\ast y)=\phi(x^\ast P_t(y))
\end{equation}
for all $x,y\in \n_\phi$ and $t\geq 0$. If $x\in \n_\phi$, then
\begin{equation*}
\phi(P_t(x)^\ast P_t(x))\leq \phi(P_t(x^\ast x))\leq \phi(x^\ast x)
\end{equation*}
by the Kadison--Schwarz inequality, so that both sides of (\ref{eq:GNS-sym}) are well-defined.

If $(P_t)$ is a $\phi$-symmetric quantum dynamical semigroup, then it commutes with the modular group $\sigma^\phi$ (the argument from \cite[Proposition 2.2]{Wir22} for a state $\phi$ and unital $(P_t)$ can be easily extended to the case when $\phi$ is a weight and $P_t$ is only assumed to be contractive).

Moreover, $(P_t)$ gives rise to a strongly continuous symmetric contraction semigroup $(P_t^{(2)})$ on $L_2(\M,\phi)$ that acts on $\Lambda_\phi(\n_\phi)$ by 
\begin{equation*}
P_t^{(2)}\Lambda_\phi(x)=\Lambda_\phi(P_t x).
\end{equation*}
Since $(P_t)$ commutes with $\sigma^\phi$, the semigroup $(P_t^{(2)})$ commutes with $(\Delta_\phi^{is})$. In particular, this definition of $(P_t^{(2)})$ is consistent with the definition in terms of the symmetric embedding $\Delta_\phi^{1/2}\Lambda_\phi$ instead of $\Lambda_\phi$. Since $(P_t)$ consists of self-adjoint maps, $(P_t^{(2)})$ also commutes with $J_\phi$.

Let $C$ be the closure of $\{\Delta_\phi^{1/4}\Lambda_\phi(x)\mid x\in \n_\phi\cap\n_\phi^\ast,\,0\leq x\leq 1\}$. We call a strongly continuous symmetric contraction semigroup $(T_t)$ on $L_2(\M,\phi)$ a \emph{symmetric quantum dynamical semigroup} if it leaves $C$ invariant. We call $(T_t)$ \emph{GNS-symmetric} if it commutes with $(\Delta_\phi^{is})$.

It follows from \cite[Theorem 4.9]{GL99} that $(P_t)\mapsto (P_t^{(2)})$ is a one-to-one correspondence between $\phi$-symmetric quantum dynamical semigroups on $\M$ and GNS-symmetric quantum dynamical semigroups on $L_2(\M,\phi)$. We call a symmetric quantum dynamical semigroup on $L_2(\M,\phi)$ a \emph{symmetric quantum Markov semigroup} if the associated quantum dynamical semigroup $(P_t)$ on $\M$ is a quantum Markov semigroup, i.e., $P_t$ is unital for all $t\geq 0$. If $\phi$ is finite, this condition is equivalent to $T_t\phi^{1/2}=\phi^{1/2}$ for all $t\geq 0$.

Next we discuss quantum Dirichlet forms. First a short detour on quadratic forms is in order. Let $H$ be a Hilbert space. There are three equivalent points of view: A quadratic form $q$ can be defined as a quadratic map from $H$ to $[0,\infty]$, a quadratic map from a subspace of $H$ to $[0,\infty)$, or as a sequilinear map from the cartesian product of a subspace of $H$ with itself to $\IC$. The first two viewpoints are related by restricting $q$ to the elements where it is finite, and the second and third are related by the polarization identity. We will use all three viewpoints interchangeably and use the same symbol to denote all three objects.

A quadratic form $q$ on $H$ is called \emph{densely defined} if $D(q)=\{\xi\in H\mid q(\xi)<\infty\}$ is dense, \emph{closable} if the map $q\colon D(q)\to [0,\infty)$ is lower semicontinuous and \emph{closed} if the map $q\colon H\to [0,\infty]$ is lower semicontinuous. For every closed densely defined quadratic form $q$ there exists a positive self-adjoint operator $A$ on $H$ such that $D(A^{1/2})=D(q)$ and $\norm{ A^{1/2}\xi}^2=q(\xi)$ for all $\xi\in D(q)$.

Now let us come back to the setting of $L_2$ spaces associated with von Neumann algebras. Let $P_C$ be the metric projection onto the cone $C$ defined above. A densely defined closed quadratic form $\E$ on $L_2(\M,\phi)$ is called \emph{Dirichlet form} if $\E\circ\J=\E$ and $\E\circ P_C\leq \E$. It is a \emph{completely Dirichlet form} if for every $n\in\IN$ the amplification
\begin{equation*}
\E_n\colon L_2(\M\otimes M_n(\IC),\phi\otimes\tr_n)\to [0,\infty],\,\E_n([\xi_{ij}])=\sum_{i,j=1}^n \E(\xi_{ij})
\end{equation*}
is a Dirichlet form. Here $\tr_n$ is the normalized trace and $L_2(\M\otimes M_n(\IC),\phi\otimes\tr_n)$ is identified with $M_n(L_2(\M,\phi))$.

We call a completely Dirichlet form $\E$ \emph{GNS-symmetric} if $\E\circ\Delta_\phi^{is}=\E$ for all $s\in \IR$. It follows from \cite[Theorem 5.7]{GL99} that whenever $(T_t)$ is a symmetric quantum dynamical semigroup on $L_2(\M,\phi)$ with generator $\L$, then the quadratic form $\E$ with domain $D(\E)=D(\L^{1/2})$ and $\E(\xi)=\norm{\L^{1/2}\xi}^2$ for all $\xi\in D(\L^{1/2})$ is a GNS-symmetric completely Dirichlet form, and every GNS-symmetric completely Dirichlet form arises in this ways.

If $\phi$ is finite, the characterization of the quadratic forms associated with GNS-symmetric quantum Markov semigroups is also easy -- the additional condition is simply $\E(\phi^{1/2})=0$. In the case when $\phi$ is a weight, such a characterization seems to be missing in the literature. We give here a necessary and sufficient criterion that is analogous to the one in the classical case from \cite[Theorem I.1.6.6]{FOT11}.

\begin{proposition}\label{prop:char_Dir_form_conservative}
Let $\M$ be a von Neumann algebra, $\phi$ a normal semi-finite faithful weight and $\E$ a GNS-symmetric completely Dirichlet form on $L_2(\M,\phi)$. The associated quantum dynamical semigroup $(P_t)$ on $\M$ is a quantum Markov semigroup if and only if there exists a sequence $(\xi_n)$ in $D(\E)\cap\Lambda_\phi'(\n_\phi^\ast)$ such that $\pi_r(\xi_n)\to 1$ strongly and
\begin{equation*}
\E(\xi_n,\eta)\to 0
\end{equation*}
for all $\eta\in D(\E)\cap \Lambda_\phi(\m_\phi)$.
\end{proposition}
\begin{proof}
Write $\L$ for $\sigma$-weak generator of $(P_t)$ on $\M$. First assume that $(P_t)$ is a quantum Markov semigroup. Let $(x_n)$ be a sequence in $\n_\phi^\ast$ such that $x_n\to 1$ strongly$^\ast$. Since $(\L+1)^{-1}$ is normal, self-adjoint and $\L(1)=0$, we have $(\L+1)^{-1}(x_n)\to 1$ strongly$^\ast$. Moreover, $(\L+1)^{-1}$ leaves $\n_\phi^\ast$ invariant by the Kadison--Schwarz inequality.

Let $\xi_n=\Lambda_\phi^\prime((\L+1)^{-1}x_n)$. If $\eta\in D(\E)\cap \Lambda_\phi(\m_\phi)$, then
\begin{equation*}
\E(\xi_n,\eta)=\langle \L^{(2)}(\L^{(2)}+1)^{-1}\Lambda_\phi^\prime(x_n),\eta\rangle=\langle \Lambda_\phi^\prime(x_n)-\Lambda_\phi^\prime((\L+1)^{-1}(x_n)),\eta\rangle.
\end{equation*}
Since $\eta\in \Lambda_\phi(\m_\phi)$, there exist $y_1,\dots,y_m,z_1,\dots,z_m\in \n_\phi$ such that $\eta=\sum_j \Lambda_\phi(y_j^\ast z_j)$. Thus
\begin{equation*}
\E(\xi_n,\eta)=\sum_{j=1}^m \langle\Lambda_\phi(y_j)\cdot(x_n-(\L+1)^{-1}(x_n)),\Lambda_\phi(z_j)\rangle\to 0.
\end{equation*}
For the converse let $\eta=\Lambda_\phi((\L+1)^{-1}(x))$ with $x\in \m_\phi$. Since $x$ is a linear combination of four elements from $\mathfrak p_\phi$ and $\phi\circ(\L+1)^{-1}\leq \phi$, we have $(\L+1)^{-1}(x)\in \m_\phi$. Moreover, writing $\L^{(2)}$ for the generator of $(P_t^{(2)})$, we have $\eta=(\L^{(2)}+1)^{-1}\Lambda_\phi(x)\in D(\E)$.

Let $(\xi_n)$ be a sequence as in the proposition and set $x_n=J\pi_r(\xi_n)^\ast J$, so that $\xi_n=\Lambda_\phi^\prime(x_n)$. Note that $\pi_r(\xi_n)\to 1$ strongly implies $x_n^\ast\to 1$ strongly. By the GNS-symmetry of $(P_t)$ we have
\begin{align*}
\E(\xi_n,\eta)&=\langle \xi_n,\Lambda_\phi(x)-\Lambda_\phi((\L+1)^{-1}(x))\rangle\\
&=\langle \xi_n-(\L^{(2)}+1)^{-1}\xi_n,\Lambda_\phi(x)\rangle\\
&=\langle \Lambda_\phi^\prime(x_n-(\L+1)^{-1}(x_n)),\Lambda_\phi(x)\rangle.
\end{align*}
Write $x=\sum_j y_j^\ast z_j$ with $y_1,\dots, y_m,z_1,\dots,z_m\in \n_\phi$ to get
\begin{equation*}
\E(\xi_n,\eta)=\sum_{j=1}^m \langle\Lambda_\phi(y_j)\cdot(x_n-(\L+1)^{-1}(x_n)),\Lambda_\phi(z_j)\rangle.
\end{equation*}
Taking the limit on both sides we obtain
\begin{equation*}
0=\sum_{j=1}^m \langle (1-(\L+1)^{-1}(1))\Lambda_\phi(y_j),\Lambda_\phi(z_j)\rangle.
\end{equation*}
As this holds for all $y_j,z_j\in \n_\phi$, we conclude $(\L+1)^{-1}(1)=1$, which implies that $1\in D(\L)$ and $\L(1)=0$.
\end{proof}

\begin{definition}
Let $\M$ be a von Neumann algebra and $\phi$ a normal semi-finite faithful weight on $\M$. We say a completely Dirichlet form $\E$ on $L_2(\M,\phi)$ is \emph{quantum Dirichlet form} if there exists a sequence $(\xi_n)$ in $D(\E)\cap\Lambda_\phi^\prime(\n_\phi^\ast)$ such that $\pi_r(\xi_n)\to 1$ strongly and $\E(\xi_n,\eta)\to 0$ for all $\eta\in D(\E)\cap\Lambda_\phi(\m_\phi)$.
\end{definition}
Thus, a GNS-symmetric completely Dirichlet form is a quantum Dirichlet form if and only if the associated semigroup is a quantum Markov semigroup.

We also need some density properties for the domain of completely Dirichlet forms. We start with an abstract approximation lemma in Hilbert spaces (see \cite[Lemma 4.9]{AGS14b} for a similar result).

\begin{lemma}\label{lem:approx}
Let $H$ be a Hilbert space, $A$ a positive self-adjoint operator on $H$ and $\beta\geq 0$. If $V\subset H$ is a dense subspace, then $\operatorname{lin}\bigcup_{t>0}e^{-tA}(V)$ is a core for $A^\beta$.
\end{lemma}
\begin{proof}
First note that $e^{-tA}(V)\subset D(A^\beta)$ by the spectral theorem. Let $\xi\in D(A^\beta)$. By assumption there exists a sequence $(\xi_n)$ in $V$ such that $\xi_n\to \xi$. By the spectral theorem, there exists $C>0$ such that
\begin{equation*}
\norm{A^\beta(e^{-tA}\xi_n-e^{-tA}\xi)}\leq \frac{C}{t^\beta}\norm{\xi-\xi_n}.
\end{equation*}
Thus $e^{-tA}\xi_n\to e^{-tA}\xi$ in the graph norm of $A^\beta$ for every $t>0$. Moreover,
\begin{equation*}
A^\beta(e^{-tA}\xi-\xi)=e^{-tA}A^\beta\xi-A^{\beta}\xi\to 0
\end{equation*}
as $t\searrow 0$. Hence $\bigcup_{t>0}e^{-tA}(V)$ is dense in $D(A^\beta)$ with respect to the graph norm.
\end{proof}

\begin{remark}
The same result holds when $A$ is the generator of an analytic contraction semigroup on a Banach space, one just has to replace the application of the spectral theorem by holomorphic functional calculus.
\end{remark}

As an application we immediately get the following density result. Recall that $\AA_\phi$ is the maximal Tomita algebra associated with $\Lambda_\phi(\n_\phi\cap\n_\phi^\ast)$.

\begin{lemma}\label{lem:bounded_elements_form_core}
Let $\M$ be a von Neumann algebra and $\phi$ a normal semi-finite faithful weight on $\M$. If $\E$ is a GNS-symmetric completely Dirichlet form on $L_2(\M,\phi)$, then $\AA_\phi\cap D(\E)$ is a core for $\E$.
\end{lemma}
\begin{proof}
Since the associated semigroup $(T_t)$ is a GNS-symmetric quantum dynamical semigroup, $\AA_\phi\cap D(\E)$ is invariant under $(T_t)$. Now the result follows from \Cref{lem:approx}.
\end{proof}

We also need a density result inside $\M$. We write $\M_1$ for the closed unit ball in $\M$.

\begin{lemma}\label{lem:bounded_elements_strong_dense}
Let $\M$ be a von Neumann algebra and $\phi$ a normal semi-finite faithful weight on $\M$. If $(T_t)$ is a GNS-symmetric quantum dynamical semigroup on $L_2(\M,\phi)$ with generator $\L^{(2)}$, then $\{x\in \n_\phi\cap\n_\phi^\ast\mid \Lambda_\phi(x)\in \AA_\phi\cap D(\E)\}\cap \M_1$ is strong$^\ast$ dense in $\M_1$.
\end{lemma}
\begin{proof}
By standard approximation arguments, the set $D=\{x\in \n_\phi\cap \n_\phi^\ast\mid \Lambda_\phi(x)\in\AA_\phi\}\cap\M_1$ is strong$^\ast$ dense in $\M_1$. Let $x\in D$ be self-adjoint and write $(P_t)$ for the quantum dynamical semigroup on $\M$ associated with $(T_t)$. By \Cref{lem:QMS_strong_conv} we have $P_t (x)^\ast=P_t x\to x$ strongly as $t\searrow 0$.

By the spectral theorem and the definition of $(T_t)$,
\begin{equation*}
\Lambda_\phi(P_t(x))=T_t\Lambda_\phi(x)\in D(\E).
\end{equation*}
Finally, $P_t (x)\in D$ since $(P_t)$ is contractive and GNS-symmetric.
\end{proof}

\section{Tomita bimodules and the Fock space construction}\label{sec:Tomita_bimodule}

In this section we introduce the notion of Tomita bimodules (\Cref{def:Tomita_bimodule}). Tomita bimodules are bimodules over Tomita algebras with certain extra structure that resembles the modular group and modular conjugation on a Tomita algebra. Tomita bimodules will serve as codomains of the derivations we associate with GNS-symmetric quantum Markov semigroups.

The resemblance of the extra structure on a Tomita bimodule with the modular group and modular conjugation on a Tomita algebra is more than just an analogy. We show (\Cref{thm:Fock_weight}) using a Fock space construction that if the actions of a Tomita algebra $\AA$ on a Tomita bimodule are normal, this bimodule can be embedded into the GNS Hilbert space of a weight on a von Neumann algebra containing $\pi_l(\AA)^{\prime\prime}$, this weight extends the canonical weight on $\pi_l(\AA)^{\prime\prime}$ and its modular group leaves $\pi_l(\AA)^{\prime\prime}$ invariant. This construction generalizes Shlyakhtenko's free Araki--Woods factors \cite{Shl97}, which correspond to the case $\AA=\IC$ (up to a change of sign).

\subsection{Tomita bimodules}

Some notation: If $\H$ is a pre-Hilbert space, we write $B(\H)$ for the set of all adjointable linear operators on $\H$. Of course this set coincides with the set of bounded linear operators if $\H$ is complete. If $A$ is a $\ast$-algebra, we write $A^\op$ for the algebra with the same vector space structure and involution as $A$ and the product $a\cdot_\op b=ba$.

\begin{definition}\label{def:Tomita_bimodule}
Let $\AA$ be a Tomita algebra. A \emph{Tomita bimodule} over $\AA$ is a pre-Hilbert space $\H$ with non-degenerate commuting $\ast$-representations $L\colon (\AA,^\sharp)\to B(\H)$, $R\colon (\AA^\op,^\flat)\to B(\H)$, an anti-linear isometric involution $\J\colon \H\to\H$ and a group $(\U_z)_{z\in\IC}$ of adjointable operators on $\H$ such that
\begin{enumerate}[(a)]
\item $\norm{L(a)}\leq \norm{\pi_l(a)}$, $\norm{R(a)}\leq \norm{\pi_r(a)}$ for all $a\in\AA$,
\item $\J L(a)=R(Ja)\J$ for all $a\in\AA$,
\item the map $\IC\to\IC,\,z\mapsto\langle \xi,\U_z \eta\rangle$ is analytic for all $\xi,\eta\in\H$,
\item $\langle \xi,\U_z \eta\rangle=\langle \U_{-\bar z}\xi,\eta\rangle$ for all $\xi,\eta\in \H$, $z\in\IC$,
\item $\U_z L(a)\U_{-z}=L(U_z a)$ for all $a\in\AA$, $z\in\IC$,
\item $\U_z \J=\J \U_{\bar z}$ for all $z\in\IC$.
\end{enumerate}
\end{definition}

\begin{remark}
One could also formulate the axioms of a Tomita bimodules also in terms of the involution $\J\U_{-i/2}$ instead of $\J$, which would make the analogy to Tomita algebras even more apparent. However, since $\J$ is bounded in contrast to $\J\U_{-i/2}$, it is easier to construct it in concrete examples. Moreover, as there is no (internal) multiplication on a Tomita bimodule, the operator $\J\U_{-i/2}$ does not have the same distinguished role as $S_0=JU_{-i/2}$ has for a Tomita algebra. For this reason we chose the presentation in terms of $\J$.
\end{remark}

\begin{remark}
In addition to the axioms of a Tomita bimodule, one could also require that the elements $\xi\in\H$ be bounded vectors (see below). However, this is typically not the case for the range of the derivations we associate with GNS-symmetric quantum Markov semigroups, so we do not include this condition in our definition of Tomita bimodules.
\end{remark}

\begin{remark}
During the preparation of this manuscript, we learned that the term ``Tomita bimodule'' has already been used for a related, but different object in \cite{GY19}. There the authors study $\AA$-bimodules with an $\AA$-valued inner product over an \emph{unimodular} Tomita algebra $\AA$ and some extra structure similar to the one in our definition.
\end{remark}

Let $\overline\H$ be the completion of $\H$. Property (a) in \Cref{def:Tomita_bimodule} guarantees that the maps $L(a)$ and $R(a)$ can be extended to bounded operators on $\overline \H$ for all $a\in\AA$. Moreover, $\J$ extends to an anti-unitary involution on $\overline\H$ and $(\U_t)_{t\in\IR}$ extends to a strongly continuous unitary group on $\overline\H$. We shall denote these extension by the same symbols. 

Note that the maps $\U_z$ for $z\in\IC\setminus \IR$ are usually not bounded on $\H$, so that they cannot be extended to bounded operators on $\overline \H$. However, they have (densely defined) closed extensions, as we will see next.

\begin{lemma}\label{lem:mod_group_closable}
Let $\AA$ be a Tomita algebra and $\H$ a Tomita bimodule over $\AA$. For every $z\in\IC$ the operator $\U_z$ is closable in $\overline\H$.
\end{lemma}
\begin{proof}
Since $(\U_t)$ is a strongly continuous unitary group on $\overline\H$, there exists an injective positive self-adjoint operator $A$ on $\overline\H$ such that $\U_t=A^{it}$. By \cite[Lemma VI.2.3]{Tak03} every $\xi\in\H$ belongs to $D(A^\alpha)$ for all $\alpha\in \IR$ and $\U_{-i\alpha}\xi=A^\alpha\xi$. Thus $A^\alpha$ is a closed extension of $\U_{-i\alpha}$.

Since $A^{it}$ is unitary for $t\in\IR$, it follows that $A^{\alpha+it}$ is closed and 
\begin{equation*}
A^{\alpha+it}\xi=A^\alpha A^{it}\xi=\U_{-i\alpha}\U_t\xi=\U_{-i\alpha+t}\xi
\end{equation*}
for all $\xi\in \H$. Hence $A^{\alpha+it}$ is a closed extension of $\U_{-i\alpha+t}$.
\end{proof}

We call the Tomita bimodule $\H$ \emph{normal} if the map $\pi_l(a)\mapsto L(a)$ extends to a normal $\ast$-representation $\tilde L$ of $\pi_l(\AA)^{\prime\prime}$ on $\overline\H$. In this case, the map $J\pi_r(a)^\ast J\mapsto R(a)$ extends to a normal $\ast$-representation $\tilde R$ of $(\pi_l(\AA)^{\prime\prime})^\op$ on $\overline\H$ (cf. \Cref{sec:carre_du_champ}). This gives rise to the closely related notion of Tomita correspondences.

\begin{definition}
Let $\M$ be a von Neumann algebra and $\phi$ a normal semi-finite faithful weight on $\M$. A \emph{Tomita $\M$-$\M$-correspondence} $\H$ is an $\M$-$\M$-correspondence $\H$ endowed with an anti-unitary involution $\J\colon\H\to\H$ and a strongly continuous unitary group $(\U_t)_{t\in\IR}$ on $\H$ such that
\begin{enumerate}[(a)]
\item $\J(x\xi y)=y^\ast(\J\xi)x^\ast$ for all $x,y\in\M$, $\xi\in\H$,
\item $\U_t(x\xi y)=\sigma^\phi_t(x)(\U_t\xi)\sigma^\phi_t(y)$ for all $x,y\in\M$, $\xi\in \H$, $t\in\IR$,
\item $\J\U_t=\U_t \J$ for all $t\in\IR$.
\end{enumerate}
\end{definition}

Condition (a) means that $\J$ implements an anti-isomorphism between $\H$ and the dual correspondence $\H^\op$.
\begin{lemma}\label{lem:boundedness_Tomita_correspondence}
A vector $\xi\in \H$ is left-bounded if and only if $\J\xi$ is right-bounded, and in this case $R_\phi(\J\xi)=\J L_\phi(\xi)J$. Moreover, if $\xi\in\H$ is left-bounded, then $\U_t\xi$ is left-bounded for all $t\in\IR$, and
\begin{equation*}
L_\phi(\U_t\xi)^\ast L_\phi(\U_t\eta)=\sigma^\phi_t(L_\phi(\xi)^\ast L_\phi(\eta))
\end{equation*}
for all left-bounded vectors $\xi,\eta\in\H$ and $t\in\IR$.
\end{lemma}
\begin{proof}
For $x\in\n_\phi\cap\n_\phi^\ast$ we have $x\J\xi=\J(\xi x^\ast)$. Since $\J$ is isometric and $\phi^{1/2} x^\ast=J(x\phi^{1/2})$, both statements for the conjugation operator follow. The statements for the group $(\U_t)$ are a direct consequence of (b).
\end{proof}

As discussed above, every normal Tomita bimodule gives rise to a Tomita correspondence. Conversely, every Tomita correspondence contains a Tomita bimodule, as we show in the next proposition. Note however that these two operations are in general not inverse to each other.

\begin{proposition}\label{prop:bounded_Tomita_bimodule}
Let $\M$ be a von Neumann algebra and $\phi$ a normal semi-finite faithful weight on $\M$. Let $\H$ be a Tomita $\M$-$\M$-correspondence and let $\H_0$ be the set of all vectors $\xi$ in $\H$ for which
\begin{equation*}
\IR\to \H,\,t\mapsto\U_t\xi
\end{equation*}
has an entire continuation $z\mapsto \U_z \xi$ and $\U_z\xi$ is bounded for all $z\in\IC$. Then $\H_0$ is dense in $\H$ and equipped with the restriction of $\J$ and $(\U_z)$ and the left and right actions
\begin{align*}
L&\colon\AA_\phi\to B(\H),\,L(a)\xi=\pi_l(a)\cdot \xi,\\
R&\colon\AA_\phi\to B(\H),\,R(a)\xi=\xi\cdot J\pi_r(a)^\ast J
\end{align*}
it forms a normal Tomita bimodule over $\AA_\phi$.
\end{proposition}
\begin{proof}
By \cite[Theorem 1]{OOT17} the set of bounded vectors is dense in $\H$. For a bounded vector $\xi\in\H$ let
\begin{equation*}
\xi_n=\sqrt{\frac n \pi}\int_\IR e^{-n t^2}\U_t\xi\,dt.
\end{equation*}
By \cite[Lemma VI.2.4]{Tak03} we have $\xi_n\to\xi$. Moreover, it is not hard to see that
\begin{equation*}
\IC\to \H,\,z\mapsto\U_z \xi=\sqrt{\frac n \pi}\int_\IR e^{-n(t-z)^2}\U_t\xi\,dt
\end{equation*}
is an entire continuation of $t\mapsto \U_t \xi_n$ (compare \cite[Theorem VI.2.2]{Tak03}).

Furthermore,
\begin{align*}
\norm{(\U_z\xi_n) x}&\leq \sqrt{\frac n\pi}\int_\IR e^{-n(t-\Re z)^2}e^{n(\Im z)^2}\norm{(\U_t \xi)x}\,dt\\
&=\sqrt{\frac n\pi}e^{n(\Im z)^2}\int_\IR e^{-n (t-\Re z)^2}\norm{\U_t(\xi \sigma^\phi_{-t}(x))}\,dt\\
&\leq \sqrt{\frac n\pi}e^{n(\Im z)^2}\int_\IR e^{-n(t-\Re z)^2}\norm{L_\phi(\xi)}\norm{\phi^{1/2}\sigma^\phi_{-t}(x)}\,dt\\
&=e^{n(\Im z)^2}\norm{L_\phi(\xi)}\norm{\phi^{1/2}x}
\end{align*}
so that $\U_z \xi_n$ is left-bounded. An analog computation shows that $\U_z \xi_n$ is right-bounded as well. This establishes the density of $\H_0$.

If $a\in\AA_\phi$ and $\xi\in \H$ is bounded, then $L(a)\xi=\pi_l(a)\cdot\xi$ is left-bounded. Moreover, if $x\in \n_\phi$, then $x\pi_l(a)\in \n_\phi$ and
\begin{align*}
\norm{x\cdot L(a)\xi}&=\norm{x\cdot \pi_l(a)\xi}\\
&\leq \norm{R_\phi(\xi)}\norm{x\pi_l(a)\phi^{1/2}}\\
&\leq \norm{R_\phi(\xi)}\norm{\pi_l(U_{i/2}(a))}\norm{x\phi^{1/2}}.
\end{align*}
Thus $L(a)\xi$ is right-bounded, as well. Similarly one shows that $R(a)\xi$ is bounded.

Furthermore, if $z\mapsto \U_z\xi$ is an entire analytic continuation of $t\mapsto \U_t \xi$, then $z\mapsto L(U_z a) R(U_z b)\U_z\xi$ is an entire analytic continuation of $t\mapsto \U_t(L(a)R(b)\xi)$. Therefore $\H_0$ is an $\AA_\phi$-bimodule with the left and right actions $L$ and $R$. Properties (a)--(f) of a Tomita bimodule now follow from the properties of a Tomita correspondence together with the uniqueness of analytic continuations.
\end{proof}

\begin{remark}
If $\xi,\eta\in \H$ are bounded vectors such that $t\mapsto \U_t\xi$ and $t\mapsto \U_t\eta$ have entire analytic continuations with $\U_z\xi$, $\U_z\eta$ bounded for all $z\in\IC$, then it follows from \Cref{lem:boundedness_Tomita_correspondence} that $L_\phi(\xi)^\ast L_\phi(\eta)$ is entire analytic for $\sigma^\phi$ and 
\begin{equation*}
\sigma^\phi_z(L_\phi(\xi)^\ast L_\phi(\eta))=L_\phi(\U_{\bar z}\xi)^\ast L_\phi(\U_z\eta)
\end{equation*}
for all $z\in\IC$.
\end{remark}

\subsection{The Fock space construction}\label{subsec:Fock_construction}

Let $\H$ be a Tomita $\M$-$\M$-correspondence. We write $\H_a$ for the set of all $\xi\in \H$ (not necessarily bounded) for which the map $t\mapsto\U_t\xi$ has an entire continuation. On $\H_a$ we define involutions $\mathcal S_0$ and $\F_0$ by $\mathcal S_0\xi=\J\U_{-i/2}\xi$ and $\F_0\xi=\J\U_{i/2}\xi$. These maps are densely defined in $\overline\H$ by \Cref{prop:bounded_Tomita_bimodule} and closable by \Cref{lem:mod_group_closable}. We denote their closures by $\mathcal S$ and $\F$, respectively. In analogy to Tomita algebras, we also write $\xi^\sharp$ for $\mathcal S\xi$ and $\xi^\flat$ for $\mathcal F \xi$.

The Fock space over $\H$ is defined as
\begin{equation*}
\F_\M(\H)=L_2(\M,\phi)\oplus\bigoplus_{n=1}^\infty \H^{\otimes_\phi n},
\end{equation*}
where $\otimes_\phi$ denotes the relative tensor product of correspondences.

For a left-bounded vector $\xi\in\H$ define the operator $a_0(\xi)$ on the algebraic direct sum by
\begin{align*}
a_0(\xi)\phi^{1/2}x&=\xi x,\\
a_0(\xi)(\xi_1\otimes_\phi\dots\otimes_\phi\xi_n)&=\xi\otimes_\phi\xi_1\otimes_\phi\dots\otimes_\phi\xi_n.
\end{align*}
By the definition of the relative tensor product, $a_0(\xi)$ is bounded with norm $\norm{L_\phi(\xi)}$. Let $a(\xi)$ denote the bounded extension of $a_0(\xi)$ to $\F_\M(\H)$ and $s(\xi)=a(\xi)+a(\xi)^\ast$.

Likewise, for a right-bounded vector $\xi\in\H$ define the operator $b_0(\xi)$ on the algebraic direct sum by
\begin{align*}
b_0(\xi)x\phi^{1/2}&=x\xi,\\
b_0(\xi)(\xi_1\otimes_\phi\dots\otimes_\phi\xi_n)&=\xi_1\otimes_\phi\dots\otimes_\phi\xi_n\otimes_\phi\xi.
\end{align*}
Again, $b_0(\xi)$ is bounded with norm $\norm{R_\phi(\xi)}$. Let $b(\xi)$ denote the bounded extension of $b_0(\xi)$ to $\F_\M(\H)$ and $t(\xi)=b(\xi)+b(\xi)^\ast$.

\begin{lemma}\label{lem:commutant_Gaussian_functor}
If $\xi\in \H$ is left-bounded with $\mathcal S\xi=\xi$ and $\eta\in \H$ is right-bounded with $\mathcal F\eta=\eta$, then $s(\xi)$ and $t(\eta)$ commute.
\end{lemma}
\begin{proof}
Since $D(\mathcal S)\cap \H_a$ and $D(\F)\cap \H_a$ are cores for $\mathcal S$ and $\F$, respectively, we may assume $\xi,\eta\in \H_a$.

If $\xi$ is left-bounded, then $\phi(L_\phi(\xi)^\ast L_\phi(\xi))<\infty$ by \cite[Lemma IX.3.3]{Tak03}. Hence, if $\xi,\eta$ are left-bounded, then
\begin{equation*}
\phi(L_\phi(\eta)^\ast L_\phi(\xi) L_\phi(\xi)^\ast L_\phi(\eta))\leq \norm{L_\phi(\xi)}^2\phi(L_\phi(\eta)^\ast L_\phi(\eta))<\infty.
\end{equation*}
Thus $L_\phi(\xi)^\ast L_\phi(\eta)\in\n_\phi$. Similarly one shows $L_\phi(\xi)^\ast L_\phi(\eta)\in \n_\phi^\ast$. If $(y_n)$ is a sequence of self-adjoint elements in $\n_\phi^\ast$ that converges strongly to $1$, then
\begin{align*}
L_\phi(\xi)^\ast L_\phi(\eta)\phi^{1/2}=\lim_{n\to\infty}L_\phi(\xi)^\ast L_\phi(\eta)\phi^{1/2} y_n=\lim_{n\to\infty}L_\phi(\xi)^\ast \eta y_n=L_\phi(\xi)^\ast \eta.
\end{align*}

If $x\in\n_\phi\cap\n_\phi^\ast$ is entire analytic for $\sigma^\phi$, then
\begin{align*}
t(\eta) s(\xi)\phi^{1/2}x&=t(\eta)\xi x\\
&=\xi x\otimes_\phi\eta+R_\phi(\eta)^\ast \xi x\\
&=\xi \otimes_\phi\sigma^\phi_{-i/2}(x)\eta+J L_\phi(\J\eta)^\ast \J (\xi x)\\
&=a(\xi)t(\eta)\sigma^\phi_{-i/2}(x)\phi^{1/2}+J L_\phi(\J\eta)^\ast x^\ast L_\phi(\J\xi)\phi^{1/2}\\
&=a(\xi)t(\eta)\sigma^\phi_{-i/2}(x)\phi^{1/2}+\phi^{1/2}(\J\xi|x \J\eta)\\
&=a(\xi)t(\eta)\sigma^\phi_{-i/2}(x)\phi^{1/2}+\sigma_{-i/2}^\phi((\J\xi|(x\J\eta))\phi^{1/2}\\
&=a(\xi)t(\eta)\sigma^\phi_{-i/2}(x)\phi^{1/2}+(\U_{i/2}\J\xi|\sigma^\phi_{-i/2}(x)\U_{-i/2}\J\eta)\phi^{1/2}\\
&=a(\xi)t(\eta)\sigma^\phi_{-i/2}(x)\phi^{1/2}+a(\xi)^\ast \sigma^\phi_{-i/2}(x)L(\eta)\phi^{1/2}\\
&=a(\xi)t(\eta)\sigma^\phi_{-i/2}(x)\phi^{1/2}+a(\xi)^\ast t(\eta)\sigma^\phi_{-i/2}(x)\phi^{1/2}\\
&=s(\xi)t(\eta)\phi^{1/2}x.
\end{align*}
By approximation, it follows that $t(\eta)$ and $s(\xi)$ commute on $L_2(\M,\phi)$. Commutation on $\H^{\otimes_\phi n}$ for $n\geq 1$ is easy to see.
\end{proof}

The Fock space $\F_\M(\H)$ inherits the structure of an $\M$-$\M$-correspondence from $\H$

\begin{definition}\label{def:Gaussian_functor}
The free Gaussian algebra $\Gamma_\M(\H)$ over $\H$ is the von Neumann algebra generated by the left action of $\M$ on $\F_\M(\H)$ and $\{s(\xi)\mid \xi\in L_\infty(\H_\M,\phi)\cap D(\mathcal S),\,\mathcal S\xi=\xi\}$.
\end{definition}

\begin{example}
A Tomita $\IC$-$\IC$-correspondence is a Hilbert space $\H$ with an anti-unitary involution $\J$ and a strongly continuous unitary group $(\U_t)$ on $\H$ that commutes with $\J$. The set of $\J$-real elements, i.e. the vectors $\xi\in \H$ with $\J\xi=\xi$, form a real Hilbert space $H$, and $(\U_t)$ restricts to a strongly continuous  orthogonal group $(V_t)$ on $H$. The resulting free Gaussian algebra $\Gamma_\M(\H)$ is the free Araki--Woods factor $\Gamma(H,(V_{-t}))^{\prime\prime}$ from \cite{Shl97}. See Subsection \ref{subsec:Araki-Woods} for a more detailed discussion.
\end{example}

Let $I\colon L_2(\M,\phi)\to \F_\M(\H)$ be the inclusion map. Since $\Gamma_\M(\H)$ consists of right module maps, $I^\ast x I\in \M$ for every $x\in \Gamma_\M(\H)$. Let
\begin{align*}
E&\colon \Gamma_\M(\H)\to\M,\,x\mapsto I^\ast x I,\\
\hat\phi&\colon \Gamma_\M(\H)_+\to \IC,\,\hat \phi(x)=\phi(E(x)).
\end{align*}
Clearly, $E$ is a normal conditional expectation and $\hat\phi$ is a normal weight. We will show next that $\hat \phi$ is semi-finite and faithful and identify its semi-cyclic representation. To do so, we need the following fact about bimodules.

\begin{lemma}[{\cite[Lemma IX.3.3]{Tak03}}]\label{lem:bdd_vec_module_map}
If $(\N,\psi)$ is a weighted von Neumann algebra and $\mathcal K$ an $\N$-$\N$-correspondence, then the map
\begin{equation*}
L_\psi\colon L_\infty(\mathcal K_\N,\psi)\to \{x\in \L(L_2(\N,\psi)_\N,\mathcal K_\N)\mid \psi(x^\ast x)<\infty\}
\end{equation*}
is bijective and $L_\psi^{-1}(xy)=x L_\psi^{-1}(y)$ for $x,y\in \L(L_2(\N,\psi)_\N,\mathcal K_\N)$ with $\psi(y^\ast y)<\infty$.
\end{lemma}

\begin{lemma}
If $\xi\in \H$ is left-bounded, then $\phi(L_\phi(\xi)^\ast L_\phi(\xi))=\norm{\xi}^2$.
\end{lemma}
\begin{proof}
Let $\xi\in \H$ be left-bounded. By [TakII, Lemma IX.3.3] there exist a bounded right module map $x\colon L_2(\M)\to\H$ and $y\in \n_\phi$ such that $\xi=xy\phi^{1/2}$. Consequently $L_\phi(\xi)=xy$ and we obtain
\begin{equation*}
\norm{\xi}^2=\norm{xy\phi^{1/2}}^2=\langle y\phi^{1/2},x^\ast x y\phi^{1/2}\rangle=\phi(y^\ast x^\ast x y)=\phi(L_\phi(\xi)^\ast L_\phi(\xi)).\qedhere
\end{equation*}
\end{proof}

\begin{theorem}\label{thm:Fock_weight}
The weight $\hat\phi$ is semi-finite and faithful and the associated semi-cyclic representation is given by $H=\F_\M(\H)$, $\pi=\mathrm{id}$ and $\Lambda(x)=L_\phi^{-1}(xI)$ for $x\in \n_{\hat\phi}$. In particular, the conditional expectation $E$ is faithful.
\end{theorem}
\begin{proof}
Write $\pi_l\colon \M\to B(\F_\M(\H))$ for the left action. If $x\in \n_\phi$, then
\begin{equation*}
\hat \phi(\pi_l(x)^\ast \pi_l(x))=\phi(I^\ast \pi_l(x^\ast x) I)=\phi(x^\ast x)<\infty.
\end{equation*}
Moreover, if $\xi\in L_\infty(\H_\M,\phi)\cap \H_a$ with $\mathcal S_0\xi=\xi$, then
\begin{equation*}
\hat\phi(s(\xi)^\ast s(\xi))=\phi(I^\ast s(\xi)^2 I)=\phi(L_\phi(\xi)^\ast L_\phi(\xi))<\infty.
\end{equation*}
Since $\n_{\hat\phi}$ is an algebra, this implies that the algebra generated by $\pi_l(\n_\phi)\cup \{s(\xi)\mid \xi\in L_\infty(\H_\M,\phi),\,\mathcal S_0\xi=\xi\}$ is contained in $\n_{\hat\phi}$. As the latter is dense in $\Gamma_\M(\H)$, this implies that $\hat\phi$ is semi-finite.

We show that $\hat\phi$ is faithful by proving that the associated semi-cyclic representation has the claimed form and the map $\Lambda$ is injective.

First note that by definition if $x\in\n_{\hat\phi}$, then $x I$ is a right module map and 
\begin{equation*}
\phi((xI)^\ast (xI))=\hat\phi(x^\ast x)<\infty.
\end{equation*}
Thus $\Lambda$ is well-defined.

Next we show that $\Lambda$ has dense image. If $x\in \n_\phi$, then
\begin{equation*}
x(\phi^{1/2}y)=x\phi^{1/2}y=L_\phi(x\phi^{1/2})\phi^{1/2}y
\end{equation*}
for $y\in \n_\phi^\ast$. Thus $\Lambda(\pi_l(x))=x\phi^{1/2}$. Hence $\Lambda(\n_{\hat\phi})\cap L_2(\M,\phi)$ is dense in $L_2(\M,\phi)$.

If $\xi\in \H_a$ is left-bounded and $\mathcal S_0\xi=\xi$, then
\begin{equation*}
s(\xi)\phi^{1/2}y=\xi y=L_\phi(\xi)\phi^{1/2}y
\end{equation*}
for $y\in \n_\phi^\ast$. Thus $\Lambda(s(\xi))=\xi$. For arbitrary bounded $\xi\in\H_a$ it follows that
\begin{equation*}
\xi=\frac 1 2(\xi+\mathcal S_0\xi)+\frac 1 2(\xi-\mathcal S_0\xi)=\frac 1 2\Lambda(s(\xi-\mathcal S_0\xi))-\frac i 2\Lambda(s(i(\xi-\mathcal S_0\xi))).
\end{equation*}
As the bounded vectors in $\H_a$ are dense in $\H$ by \Cref{prop:bounded_Tomita_bimodule}, it follows that $\Lambda(\n_{\hat\phi})\cap\H$ is dense in $\H$.

Now we proceed by induction to show that $\Lambda(\n_{\hat\phi})\cap\H^{\otimes_\phi n}$ is dense in $\H^{\otimes_\phi n}$. Assume it is already established for $k\leq n$. Let $\xi\in \H_a$ be left-bounded with $\mathcal S_0\xi=\xi$ and $\eta\in \H^{\otimes_\phi n}$. By induction hypothesis there are sequences $(x_n)$, $(y_n)$ in $\n_{\hat\phi}$ such that $\Lambda(x_n)\to \eta$ and $\Lambda(y_n)\to a^\ast(\xi)\eta$.

Hence
\begin{align*}
\xi\otimes\eta=s(\xi)\eta+a^\ast(\xi)\eta=\lim_{n\to\infty}s(\xi)\Lambda(x_n)+\Lambda(y_n)=\lim_{n\to\infty}\Lambda(s(\xi) x_n+y_n).
\end{align*}
As above one concludes that $\Lambda(\n_{\hat\phi})$ is dense in $\H^{\otimes_\phi (n+1)}$. Since $\bigcup_{N\geq 1}(L_2(\M,\phi)\oplus\bigoplus_{n=1}^N \H^{\otimes_\phi n})$ is dense in $\F_\M(\H)$, it follows that $\Lambda$ has dense image. Clearly, 
\begin{equation*}
\langle \Lambda(x),\Lambda(y)\rangle=\phi(L_\phi(\Lambda(x))^\ast L_\phi(\Lambda(y)))=\phi((xI)^\ast (yI))=\hat\phi( x^\ast y)
\end{equation*}
for $x,y\in\n_{\hat\phi}$.

Let $\pi_r\colon \M^\op\to \F_\M(\H)$ be the right action and $\N$ the von Neumann algebra generated by $\pi_r(\M^\op)\cup\{t(\xi)\mid \xi\in L_\infty(_\M\H,\phi)\cap \H_a,\,\F_0\xi=\xi\}$. By \Cref{lem:commutant_Gaussian_functor}, $\N\subset \Gamma_\M(\H)^\prime$. As in the previous step one can show that $\N L_2(\M,\phi)$ is dense in $\F_\M(\H)$.

If $x\in \Gamma_\M(\H)$ with $\hat\phi(x^\ast x)=0$, then $xI=0$ since $\phi$ is faithful. In other words, $x\xi=0$ for all $\xi\in L_2(\M,\phi)$. As $\N\subset \Gamma_\M(\H)^\prime$, it follows that
\begin{equation*}
0=y x\xi=x y\xi
\end{equation*}
for all $\xi\in L_2(\M,\phi)$, $y\in\N$. Since $\N L_2(\M,\phi)$ is dense in $\F_\M(\H)$, we conclude $x=0$.

Finally, if $x\in\Gamma_\M(\H)$ with $E(x^\ast x)=0$, then $0=\phi(E(x^\ast x))=\hat \phi(x^\ast x)$. This implies $x=0$ as we have just seen.
\end{proof}

\begin{remark}
Similar constructions to the one of $\Gamma_\M(\H)$ have been explored before (usually in terms of $C^\ast$-bimodules instead of correspondences), most notably in  Shlyakhtenko's work operator-valued semicircular systems \cite{Shl99} for a normal completely positive map and its extension to quantum Markov semigroups in \cite{JRS}.

The central new element in our treatment is the use of the extra structure of $\H$ in terms of $(\U_z)$ and $\J$, or equivalently, $\mathcal S$, which gives rise to a notion of ``real vectors'' in $\H$, combined with the insight the every GNS-symmetric quantum Markov semigroup gives rise to such extra structure. This ties the construction of $\Gamma_\M(\H)$ closer to the definition of the free Araki--Woods factors from \cite{Shl97}. In particular, this construction in terms of real vectors guarantees that the conditional expectation $E$ on $\Gamma_\M(\H)$ is always faithful (compare \cite[Proposition 5.2]{Shl99}).
\end{remark}

\section{Derivations with values in Tomita bimodules}\label{sec:sym_deriv}

In this section we introduce the notion of symmetric derivations with values in a Tomita bimodule (\Cref{def:sym_deriv}) and show that if a symmetric derivation is inner or can be approximated by inner derivations in a suitable sense, it gives rise to a GNS-symmetric completely Dirichlet form (\Cref{thm:from_deriv_to_QMS}).

\begin{definition}\label{def:sym_deriv}
Let $\AA$ be a Tomita algebra and $\H$ a Tomita bimodule over $\AA$. A \emph{symmetric derivation} is a linear map $\delta\colon \AA\to \H$ such that
\begin{enumerate}[(a)]
\item $\delta(ab)=a\delta(b)+\delta(a)b$ for all $a,b\in\AA$,
\item $\delta(Ja)=\J\delta(a)$ for all $a\in\AA$,
\item $\delta(U_z a)=\U_z\delta(a)$ for all $a\in\AA$, $z\in\IC$.
\end{enumerate}
A symmetric derivation is called closable (resp. bounded) if and only if it is closable (resp. bounded) as densely defined operator from the completion of $\AA$ to the completion of $\H$. A symmetric derivation $\delta$ is called \emph{inner} if there exists a bounded vector $\xi\in\H$ such that $\delta(a)=a\xi-\xi a$ for all $a\in\AA$.
\end{definition}

Recall that we introduced the notation $\xi^\sharp=\J\U_{-i/2}\xi$ and $\xi^\flat=\J\U_{i/2}\xi$ for $\xi\in \H$. With these definitions, the following lemma is immediate.

\begin{lemma}\label{lem:deriv_conjug}
Let $\AA$ be a Tomita algebra and $\H$ a Tomita bimodule. If $\delta\colon \AA\to\H$ is a symmetric derivation, then $\delta(a^\sharp)=\delta(a)^\sharp$ and $\delta(a^\flat)=\delta(a)^\flat$ for all $a\in\AA$.
\end{lemma}

If the Tomita bimodule is normal, the vector representing an inner symmetric derivation can be chosen (possibly in the completion of $\H$) so that it is invariant under $(\U_z)$ and $\J$. To prove this, we need some preparation. The key steps were already presented in \cite{Wir22}, but since the result was proven only for derivations induced by quantum Markov semigroups and the case of weights needs some technical adjustments, we give a full proof here.

Let $\H$ be an $\M$-$\M$-correspondence. The space $\L(L_2(\M)_\M,\H_\M)$ is a $C^\ast$-bimodule over $\M$ when equipped with the $\M$-valued inner product $(x,y)\mapsto x^\ast y$. In fact, since $\L(L_2(\M)_\M,\H_\M)$ is strongly closed in $B(L_2(\M),\H)$ and the left action is normal, it is a von Neumann bimodule.

Recall that a von Neumann bimodule $F$ has a unique (isometric) predual \cite[Theorem 2.6]{Sch02}. It can be realized \cite[Proposition 3.8]{Pas73} as $\overline F\otimes_\pi \M_\ast/\ker\iota$, where $\overline F$ is the Banach space with the addition and norm from $F$ and the scalar multiplication $(z,\xi)\mapsto \bar z \xi$, $\otimes_\pi$ denotes the projective tensor product and $\iota(\xi\otimes\omega)=\omega((\xi|\cdot))$. The left and right action of $\M$ on $F$ are weak$^\ast$ continuous.

\begin{lemma}\label{lem:compactness_Hilbert_module}
Let $\H$ be an $\M$-$\M$-correspondence and $\phi$ a normal semi-finite faithful weight on $\M$. For $\lambda,\mu>0$ the set
\begin{equation*}
C=\{L_\phi(\xi)\mid \xi\in L_\infty(\H_\M,\phi),\norm{\xi}\leq\lambda,\norm{L_\phi(\xi)}\leq \mu\}
\end{equation*}
is weak$^\ast$-compact.
\end{lemma}
\begin{proof}
By \Cref{lem:bdd_vec_module_map}, the set $C$ coincides with 
\begin{equation*}
\{x\in \L(L_2(\M)_\M,\H_\M)\mid \phi(x^\ast x)\leq \lambda,\,\norm{x}\leq\mu\},
\end{equation*}
which is weak$^\ast$ precompact by the Banach-Alaoglu theorem.

Moreover,
\begin{equation*}
\phi(x^\ast x)=\sup_{\substack{\omega\in \M_\ast^+\\\omega\leq \phi}}\omega(x^\ast x)=\sup_{\substack{\omega\in \M_\ast^+\\\omega\leq \phi}}\sup_{\substack{y\in \L(L_2(\M)_\M,\H_\M)\\\omega(y^\ast y)\leq 1}}\abs{\omega(x^\ast y)}^2.
\end{equation*}
By the description of the predual given above, functionals of the form $x\mapsto \omega(y^\ast x)$ with $y\in \L(L_2(\M)_\M,\H_\M)$ and $\omega\in \M_\ast$ are weak$^\ast$ continuous. Thus $x\mapsto \phi(x^\ast x)$ is weak$^\ast$ lower semicontinuous. Therefore $C$ is weak$^\ast$ closed.
\end{proof}

\begin{proposition}\label{prop:rep_vector_nice}
Let $\AA$ be a Tomita algebra and $\H$ a normal Tomita bimodule over $\AA$. If $\delta\colon \AA\to \H$ is a symmetric inner derivation, then there exists a vector $\xi\in\overline{\H}$ such that $\U_t\xi=\xi$ for all $t\in \IR$ and $\J\xi=\xi$.
\end{proposition}
\begin{proof}
Let $\M=\pi_l(\AA)^{\prime\prime}$ and $\phi$ the normal semi-finite faithful weight on $\M$ induced by $\AA^{\prime\prime}$.

For $x,y\in \L(L_2(\M)_\M,\overline\H_\M)$ and $\omega\in\M_\ast$ we have
\begin{equation*}
\omega(y^\ast \U_t x U_{-t})=\omega(U_{t}(\U_{-t}y U_t)^\ast x U_{-t}).
\end{equation*}
Since $x\mapsto U_t x U_{-t}$ leaves $\M$ invariant, $\omega(U_t\,\cdot\,U_{-t})\in\M_\ast$ and $\norm{\omega(U_t\,\cdot\,U_{-t})}=\norm{\omega}$. Moreover, $\U_{-t}y U_t\in \L(L_2(\M)_\M,\overline\H_\M)$ and $\norm{\U_{-t}y U_t}=\norm{y}$. It follows from the description of the predual given above that $x\mapsto \U_t x U_{-t}$ is a weak$^\ast$ continuous map on $\L(L_2(\M)_\M,\overline\H_\M)$.

For a bounded vector $\xi\in \overline\H$ let $\delta_\xi\colon a\mapsto a\xi-\xi a$. By assumption there exists a bounded vector $\eta\in \H$ such that $\delta=\delta_\eta$. Thus the set
\begin{equation*}
C=\{L_\phi(\xi)\mid \xi\in L_\infty(\overline\H_\M,\phi),\,\norm{\xi}\leq \norm\eta,\,\norm{L_\phi(\xi)}\leq \norm{L_\phi(\eta)},\,\delta=\delta_\xi\}
\end{equation*}
is non-empty. By \Cref{lem:compactness_Hilbert_module}, it is weak$^\ast$-compact.

By the previous discussion, $x\mapsto \U_t x U_{-t}$ is weak$^\ast$ continuous. Note that by the definition of a Tomita algebra, $\U_t L_\phi(\xi)U_{-t}=L_\phi(\U_t\xi)$. Moreover, if $L_\phi(\xi)\in C$ and $a\in \AA$, then
\begin{equation*}
\delta_{\U_t\xi}(a)=a\U_t\xi-(\U_t\xi)a=\U_t((U_{-t}a)\xi-\xi U_{-t}a)=\U_t\delta(U_{-t}a)=\delta(a).
\end{equation*}
Therefore, $x\mapsto\U_t x U_{-t}$ leaves $C$ invariant. It follows from the Ryll-Nardzewski fixed-point theorem that there exists $x\in C$ such that $\U_t x U_{-t}=x$, or, in other words, a left-bounded vector $\xi\in \overline\H$ such that $\U_t\xi=\xi$ and $\delta_\xi=\delta$.

Since there exists a bounded vector $\eta\in \H$ such that $\delta=\delta_\eta$, the derivation $\delta$ is bounded from $\AA$ to $\H$. Thus
\begin{equation*}
\norm{a\xi}\leq \norm{\delta(a)}+\norm{\xi a}\leq (\norm{\delta}+\norm{\xi})\norm{a},
\end{equation*}
which means that $\xi$ is right-bounded as well.

Finally, to ensure $\J\xi=\xi$, one can replace $\xi$ by $\frac 1 2(\xi+\J\xi)$. Thus preserves bounded and invariance under $(\U_t)$ and still represents $\delta$ by the property $\J\circ\delta=\delta\circ J$.
\end{proof}

For the following proposition note that if $\AA$ is a Tomita algebra with Hilbert completion $\HH$ and $\phi$ the associated weight on $\pi_l(\AA)^{\prime\prime}$, then $a\mapsto \pi_l(a)\phi^{1/2}$ induces a unitary operator from $\HH$ to $L_2(\pi_l(\AA)^{\prime\prime},\phi)$. Under this identification it makes sense to speak of GNS-symmetric quantum Markov semigroups on $\HH$.

\begin{proposition}\label{prop:from_bdd_deriv_to_QMS}
Let $\AA$ be a Tomita algebra with completion $\HH$ and $\H$ a normal Tomita bimodule over $\AA$. If $\delta\colon\AA\to \H$ is a symmetric inner derivation, then the quadratic form
\begin{equation*}
\E\colon\AA\to [0,\infty),\,\E_0(a)=\norm{\delta(a)}_\H^2
\end{equation*}
is bounded on $\HH$ and its closure is a GNS-symmetric quantum Dirichlet form.
\end{proposition}
\begin{proof}
Boundedness of $\E$ follows immediately from the definition of an inner derivation. By \Cref{prop:rep_vector_nice} there exists a bounded vector $\xi\in\overline\H$ such that $\U_t\xi=\xi$ for all $t\in\IR$, $\J\xi=\J$ and $\delta(a)=a\xi-\xi a$ for all $a\in\AA$. A direct computation shows 
\begin{equation*}
\E(a,b)=\langle (\xi|\xi) a+a(\xi|\xi)-2(\xi|L(a)\xi),b\rangle.
\end{equation*}
The operator
\begin{equation*}
\pi_l(\AA)^{\prime\prime}\to\pi_l(\AA)^{\prime\prime},\,x\mapsto (\xi|\xi)x+x(\xi|\xi)-2(\xi|x\xi)
\end{equation*}
is conditionally negative definite, so that it generates a quantum Markov semigroup $(P_t)$ on $\M$ by \cite[Theorem 14.7]{EL77}.

Since $\U_t\xi =\xi$ for all $t\in\IR$, we have
\begin{align*}
&\quad\;U_t((\xi|\xi)x+x(\xi|\xi)-2(\xi|x\xi))U_{-t}\\
&=(\U_t\xi|\U_t\xi)U_t x U_{-t}+U_t x U_{-t}(\U_t\xi|\U_t\xi)-2(\U_t \xi,U_t x U_{-t}\cdot\U_t\xi)\\
&=(\xi|\xi)U_t x U_{-t}+U_t xU_{-t}(\xi|\xi)-2(\xi|U_t x U_{-t}\cdot\xi).
\end{align*}
Thus the semigroup $(P_t)$ commutes with the modular group on $\pi_l(\AA)^{\prime\prime}$. Therefore the closure of $\E$ is GNS-symmetric quantum Dirichlet form.
\end{proof}

\begin{definition}
Let $\AA$ be a Tomita algebra with completion $\HH$ and $\H$ a Tomita bimodule over $\AA$. A symmetric derivation $\delta\colon\AA\to\H$ is called \emph{nearly inner} if it is closable and there exists a sequence $(\H_n)$ of normal Tomita bimodules over $\AA$ and a sequence $(\delta_n)$ of inner symmetric derivations from $\AA$ to $\H_n$ such that
\begin{equation*}
\norm{\bar\delta_n(\xi)}_{\H_n}\nearrow \begin{cases}\norm{\bar\delta(\xi)}_\H&\text{if }\xi\in D(\bar\delta),\\ \infty&\text{otherwise}.\end{cases}
\end{equation*}
\end{definition}

\begin{theorem}\label{thm:from_deriv_to_QMS}
Let $\AA$ be a Tomita algebra and $\H$ a Tomita bimodule over $\AA$. If $\delta\colon \AA\to \H$ is a nearly inner symmetric derivation, then the quadratic form
\begin{equation*}
\E_0\colon \AA\to[0,\infty),\,\E_0(a)=\norm{\delta(a)}_\H^2
\end{equation*}
is closable and its closure is an GNS-symmetric completely Dirichlet form.
\end{theorem}
\begin{proof}
Closability of $\E_0$ follows directly from the closability of $\delta$. Clearly the closure $\E$ is invariant under $(U_t)$ and $J$. If $\xi\in\HH$ and $P$ denotes the metric projection onto the closure of $\{a\in\AA\mid 0\leq\pi_l(a)\leq 1\}$, then
\begin{equation*}
\E(P(\xi))=\lim_{n\to\infty}\norm{\bar\delta_n(P(\xi))}^2\leq\liminf_{n\to\infty}\norm{\bar\delta_n(\xi)}^2\leq \E(\xi),
\end{equation*} 
where we used that $\eta\mapsto\norm{\delta_n(\eta)}^2$ is a Dirichlet form for all $n\in\IN$ by \Cref{prop:from_bdd_deriv_to_QMS}. The same argument applied to the matrix amplifications implies that $\E$ is a completely Dirichlet form.
\end{proof}

\begin{remark}
If $\AA$ is a unital Tomita algebra, then $\delta(1)=L(1)\delta(1)+R(1)\delta(1)=2\delta(1)$, hence $\delta(1)=0$. Therefore if $\delta$ is nearly inner, the completely Dirichlet form induced by $\delta$ is a quantum Dirichlet form. In general, this is not necessarily true, even in the commutative case. For example, if $(M,g)$ is a complete Riemannian manifold, $\AA=C_c^\infty(M)$ and $\delta=\nabla$, then the associated semigroup is the heat semigroup $(e^{t\Delta})$, which is not necessarily conservative (see \cite[Section 3.2]{Gri99} for example).
\end{remark}

We note for later use that while the form $\E$ is only defined in terms of inner product of elements in the range of $\delta$, it retains information about the inner product on the bimodule generated by the range of $\delta$ due to the product rule.

\begin{lemma}\label{lem:inner_prod_Tomita_bimodule}
Let $\AA$ be a Tomita algebra and $\H$ a Tomita bimodule over $\AA$. If $\delta\colon \AA\to\H$ is a symmetric derivation and $\E(a)=\norm{\delta(a)}_\H^2$, then
\begin{equation*}
\langle \delta(a)b,\delta(c)d\rangle_\H=\frac 1 2(\E(a,cdb^\flat)+\E(abd^\flat,c)-\E(bd^\flat,a^\sharp c))
\end{equation*}
for all $a,b,c,d\in\AA$.
\end{lemma}
\begin{proof}
Note that $\E(a,b)=\langle \delta(a),\delta(b)\rangle$ for all $a,b\in\AA$ by the polarization identity. Let $a,b,c,d\in\AA$. By the product rule and \Cref{lem:deriv_conjug} we have
\begin{align*}
&\quad\;\E(a,cdb^\flat)+\E(abd^\flat,c)-\E(bd^\flat,a^\sharp c)\\
&=\langle \delta(a),\delta(c)db^\flat+c\delta(db^\flat)\rangle+\langle\delta(a)bd^\flat+a\delta(bd^\flat),\delta(c)\rangle-\langle \delta(bd^\flat),\delta(a^\sharp) c+a^\sharp\delta(c)\rangle\\
&=2\langle \delta(a)b,\delta(c)d\rangle+\langle \delta(a),c\delta(db^\flat)\rangle-\langle \delta(bd^\flat),\delta(a)^\sharp c\rangle.
\end{align*}
For the last summand we have
\begin{align*}
\langle \delta(bd^\flat),\delta(a)^\sharp c\rangle&=\langle \delta(bd^\flat)c^\flat,\delta(a)^\sharp \rangle\\
&=\langle \J(\delta(a)^\sharp), \J(\delta(bd^\flat)c^\flat)\rangle\\
&=\langle \U_{-i/2}\delta(a),\U_{i/2}c\cdot \J(\delta(bd^\flat))\rangle\\
&=\langle \U_{-i/2}\delta(a),\U_{i/2}(c\delta(bd^\flat)^\flat)\rangle\\
&=\langle \delta(a),c\delta(db^\flat)\rangle.
\end{align*}
This gives the desired identity.
\end{proof}

\section{Symmetric derivations associated with quantum Dirichlet forms}\label{sec:from_QMS_to_deriv}

In this section we prove the main result of this article, namely that every GNS-symmetric quantum Dirichlet form gives rise to a Tomita bimodule and derivation (Theorems \ref{thm:from_QMS_to_deriv} and \ref{thm:deriv_unique}). In the case when the reference weight is finite, this yields a one-to-one correspondence between closable nearly inner symmetric derivations with maximal domain and GNS-symmetric quantum Dirichlet forms (\Cref{cor:bijection_deriv_QMS}).

One of the key challenges when constructing a derivation associated with a quantum Dirichlet form is that the domain of the derivation has to be an algebra for the product rule to make sense.  In the tracially symmetric case it was proven by Davies and Lindsay \cite[Proposition 3.4]{DL92} that $D(\E)\cap \M$ is an algebra, and the derivation constructed by Cipriani and Sauvageot \cite{CS03} can be defined on this algebra.

Here we propose
\begin{equation*}
\AA_\E=\{a\in\AA_\phi\mid U_z a\in D(\E)\text{ for all }z\in\IC\}
\end{equation*}
as a replacement for $D(\E)\cap \M$ in the non-tracial case. We show that $\AA_\E$ is a Tomita algebra and a core for $\E$ (\Cref{thm:Tomita_algebra_QMS}). If $\phi$ is a trace, then $\AA_\E$ coincides with $D(\E)\cap \M$.

Once this is established, the construction of the Tomita bimodule $\H$ and the derivation $\delta$ can be summarized as follows: The pre-Hilbert space $\H$ consists of linear combinations of the (formal) elements $\delta(a)b$ with the inner product determined by \Cref{lem:inner_prod_Tomita_bimodule}, the left and right action of $\AA_\E$ are the ones suggested by the product rule and $\U_z$ and $\J$ are uniquely determined by the definition of a Tomita bimodule.

Since the inner product determined by \Cref{lem:inner_prod_Tomita_bimodule} is in general degenerate, the difficulty lies in proving that all these operations are well-defined. To prove this as well as the fact that $\AA_\E$ is a Tomita algebra, we crucially rely on approximation by bounded Dirichlet forms and the construction from \cite{Wir22} (upgraded here to cover the case of weights).

\subsection{The bounded case}

In the case when $\phi$ is finite, it was shown in \cite{Wir22} that every bounded GNS-symmetric quantum Dirichlet form $\E$ can be written as
\begin{equation*}
\E(\Lambda_\phi(x),\Lambda_\phi(y))=\phi((\partial(x)|\partial(y)),
\end{equation*}
where $\partial$ is a derivation with values in a von Neumann bimodule $F$ that carries some extra structure reflecting the GNS symmetry, similar to the definition of Tomita bimodules and correspondences in the present article. In this subsection we extend this construction to the case of general weights and show how it relates to our formulation in terms of Tomita bimodules.

We will rely on the following interplay between correspondences and von Neumann bimodules (see \cite[Proposition 3.1]{Wir22} in the case of a finite weight).

\begin{proposition}
If $\M$ is a von Neumann algebra and $\H$ is an $\M$-$\M$-correspondence, then $\L(L_2(\M)_\M,\H_\M)$ with the usual left and right actions and the $\M$-valued inner product $(x\vert y)=x^\ast y$ is a von Neumann $\M$-bimodule.

Moreover, the map
\begin{equation*}
\Phi\colon \L(L_2(\M)_\M,\H_\M)\bar\odot L_2(\M)\to\H,\,x\otimes\xi\mapsto x\xi
\end{equation*}
is a unitary bimodule map.

Conversely, if $F$ is a von Neumann $\M$-bimodule, then $F\bar\odot L_2(\M)$ is an $\M$-$\M$-corres\-pondence and the map
\begin{equation*}
\Psi\colon F\mapsto \L(L_2(\M)_\M,(F\bar\odot L_2(\M))_\M),\,\xi\mapsto a(\xi)
\end{equation*}
is a bijective bimodule map that preserves the $\M$-valued inner products.
\end{proposition}
\begin{proof}
Clearly $\L(L_2(\M)_\M,\H_\M)$ is a $C^\ast$ $\M$-bimodule with the described operations. Since it is strongly dense in $B(L_2(\M),\H)$, it is a von Neumann $\M$-bimodule. We have
\begin{equation*}
\langle x\xi,y\eta\rangle=\langle \xi,x^\ast y\eta\rangle=\langle x\otimes\xi,y\otimes \eta\rangle.
\end{equation*}
Thus $\Phi$ is isometric. By \cite[Lemma IX.3.3]{Tak03} it has dense range. Thus $\Phi$ is unitary.

For the converse it is clear that $F\bar\odot L_2(\M)$ is an $\M$-$\M$-correspondence and $\Psi$ is a bimodule map that preserves the $\M$-valued inner products. It remains to show that $\Psi$ is surjective. Let $x\in \L(L_2(\M)_\M,\H_\M)$ and define
\begin{equation*}
T\colon F\to\M,\,y\mapsto x^\ast a(y).
\end{equation*}
As $T$ is a right module map and von Neumann bimodules are self-dual by \cite[Theorem 3.2.11]{Ske01}, there exists $z\in F$ such that
\begin{equation*}
x^\ast a(y)=T(y)=(z\vert y)=a(z)^\ast a(y).
\end{equation*}
Hence $x=a(z)$.
\end{proof}

\begin{proposition}\label{prop:deriv_bdd}
Let $\M$ be a von Neumann algebra and $\phi$ a normal semi-finite faithful weight on $\M$. If $\Phi\colon\M\to\M$ is a GNS-symmetric normal unital completely positive map, then there exists a normal Tomita bimodule $\H$ over $\AA_\phi$ and a symmetric derivation $\delta\colon \AA_\phi\to \H$ such that 
\begin{equation}\label{eq:deriv_bdd_QMS}
\langle \delta(a),\delta(a)\rangle_\H=\langle a,b-\Phi^{(2)}(b)\rangle
\end{equation}
for all $a,b\in\AA_\phi$. If $\phi$ is finite, then $\delta$ can be chosen inner.
\end{proposition}
\begin{proof}
First let $F$ be the GNS von Neumann bimodule (or Stinespring bimodule) associated with $\Phi$. It is constructed as follows. On $\{\sum_j x_j\otimes y_j\mid \sum_j x_j y_j=0\}$ define an $\M$-valued sesquilinear form by
\begin{equation*}
\left\langle\sum_j x_j\otimes y_j,\sum_k \tilde x_k\otimes \tilde y_k\right\rangle=\frac 1 2\sum_{j,k}y_j^\ast \Phi(x_j^\ast \tilde x_k)\tilde y_k.
\end{equation*}
Let $F_0$ be the $C^\ast$ $\M$-bimodule obtained after separation and completion with respect to $\norm{\cdot}_F=\norm{(\cdot\vert\cdot)}^{1/2}$. Then $F$ is the strong closure of $F_0$ inside $B(L_2(\M),F_0\bar\odot L_2(\M))$. We write $\sum_j x_j\otimes_\Phi y_j$ for the image of $\sum_j x_j\otimes y_j$ in $F$.

The map 
\begin{equation*}
\partial\colon \M\to F,\,x\mapsto x\otimes_\Phi 1-1\otimes_\Phi x
\end{equation*}
is a derivation and
\begin{equation*}
(\partial(x)\vert\partial(y))=\frac 1 2((I-\Phi)(x)^\ast y+x^\ast(I-\Phi)(y)-(I-\Phi)(x^\ast y)).
\end{equation*}
As in \cite[Proposition 4.4]{Wir22} one checks that the map
\begin{equation*}
\sum_j x_j\otimes_\Phi y_j\mapsto \sum_j \sigma^\phi_t(x_j)\otimes_\Phi \sigma^\phi_t(y_j)
\end{equation*}
is well-defined and extends to a weak$^\ast$ continuous linear map $V_t$ on $F$ that satisfies $V_s V_t=V_{s+t}$, $V_t(x\xi y)=\sigma^\phi_t(x)V_t \xi \sigma^\phi_t(y)$ and $(V_t\xi\vert V_t \eta)=\sigma^\phi_t((\xi\vert\eta))$. Since $\Phi$ is normal, a direct computation shows that $V_t\xi\to\xi$ strongly as $t\to 0$ if $\xi=\sum_j x_j\otimes_\Phi y_j$. For arbitrary $\xi\in F$ this follows from the strong density of elements of this form in $F$ by standard arguments.

Let $\K=F\bar\odot L_2(\M)$, $\n_\phi(F)=\{\xi\in F\mid \phi((\xi\vert\xi))<\infty\}$ and
\begin{equation*}
\Lambda\colon\n_\phi(F)\to \K,\,\xi\mapsto L_\phi^{-1}(a(\xi)).
\end{equation*}
Since $\phi((\xi\vert\xi))=\phi(a(\xi)^\ast a(\xi))$, this map is well-defined by \Cref{lem:bdd_vec_module_map}. Moreover, it has dense range by \cite[Lemma IX.3.3]{Tak03}. If $x\in \n_\phi$, then
\begin{align*}
\phi((\partial(x)\vert \partial(x)))&=\frac 1 2\phi((I-\Phi)(x)^\ast x+x^\ast(I-\Phi)(x)-(I-\Phi)(x^\ast x))\\
&=\frac 1 2\phi((I-\Phi)(x)^\ast x+x^\ast(I-\Phi)(x))\\
&=\phi(x^\ast (x-\Phi(x))\\
&=\langle x\phi^{1/2},(I-\Phi^{(2)})x\phi^{1/2}\rangle\\
&\leq \norm{I-\Phi^{(2)}}\phi(x^\ast x).
\end{align*}
Thus $\partial(x)\in \n_\phi(F)$. If we set $\delta(x\phi^{1/2})=\Lambda(\partial(x))$, then \eqref{eq:deriv_bdd_QMS} follows from the last computation by polarization.

For $\xi_1,\xi_2\in F$ and $\eta_1,\eta_2\in L_2(\M)$ we have
\begin{align*}
\langle V_t\xi_1\otimes \Delta_\phi^{it}\eta_1,V_t\xi_2\otimes\Delta_\phi^{it}\eta_2\rangle&=\langle \Delta_\phi^{it}\eta_1,(V_t\xi_1\vert V_t \xi_2)\Delta_\phi^{it}\eta_2\rangle\\
&=\langle  \Delta_\phi^{it}\eta_1,\sigma^\phi_t((\xi_1\vert\xi_2))\Delta_\phi^{it}\eta_2\rangle\\
&=\langle \eta_1,(\xi_1\vert\xi_2)\eta_2\rangle\\
&=\langle \xi_1\otimes\eta_1,\xi_2\otimes\eta_2\rangle.
\end{align*}
Thus the map $\xi\otimes\eta\mapsto V_t\xi\otimes \Delta_\phi^{it}\eta$ extends to a unitary operator $\U_t$ on $\K$. The pointwise strong continuity of $(V_t)$ and $(\Delta_\phi^{it})$ implies the strong continuity of $(\U_t)$.

If $x_j,y_j\in \n_\phi\cap \n_\phi^\ast$ are entire analytic with $\sum_j x_j y_j=0$ we have
\begin{align*}
&\quad\;\left\lVert \Lambda\left(\sum_j\sigma^\phi_{i/2}(y_j)^\ast\otimes_\Phi\sigma^\phi_{i/2}(x_j)^\ast\right)\right\rVert^2\\
&=\sum_{j,k}\phi((\sigma^\phi_{i/2}(y_j)^\ast\otimes_\Phi\sigma^\phi_{i/2}(x_j)^\ast\vert \sigma^\phi_{i/2}(y_k)^\ast\otimes_\Phi\sigma^\phi_{i/2}(x_k)^\ast))\\
&=\frac 1 2\sum_{j,k}\phi(\sigma^\phi_{i/2}(x_j)\Phi(\sigma^\phi_{i/2}(y_j)\sigma^\phi_{i/2}(y_k)^\ast)\sigma^\phi_{i/2}(x_k)^\ast)\\
&=\frac1  2\sum_{j,k}\phi(\Phi(\sigma^\phi_{i/2}(y_j)\sigma^\phi_{-i/2}(y_k^\ast))\sigma^\phi_{-i/2}(x_k^\ast x_j))\\
&=\frac 1 2\sum_{j,k}\phi(\sigma^\phi_{i/2}(y_j)\sigma^\phi_{-i/2}(y_k^\ast)\sigma^\phi_{-i/2}(\Phi(x_k^\ast x_j)))\\
&=\frac 1 2\sum_{j,k}\phi(\sigma^\phi_{-i/2}(y_k^\ast\Phi(x_k^\ast x_j)y_j))\\
&=\left\lVert \Lambda\left(\sum_j x_j\otimes_\Phi y_j\right)\right\rVert^2,
\end{align*}
where we used the GNS-symmetry of $\Phi$ and in particular the commutation with the modular group.

Thus
\begin{equation*}
\Lambda\left(\sum_j x_j\otimes_\Phi y_j\right)\mapsto\Lambda\left(\sum_j\sigma^\phi_{i/2}(y_j)^\ast\otimes_\Phi\sigma^\phi_{i/2}(x_j)^\ast\right)
\end{equation*}
is well-defined and extends to an anti-unitary operator $\J$ on $\K$. It is not hard to check that $(\U_t)$ and $\J$ make $\K$ into an $\M$-$\M$-Tomita correspondence. Clearly, $\delta(a)$ is bounded and entire analytic for $(\U_t)$ for every $a\in\AA_\phi$. Hence the Tomita bimodule $\H$ constructed from $\K$ by means of \Cref{prop:bounded_Tomita_bimodule} does the job.

By \cite[Theorem 2.1]{CE79} there exists $\xi\in F$ such that $\partial(x)=x\xi-\xi x$. If $\phi$ is finite, then $\Lambda(\xi)$ is a well-defined left-bounded vector in $\K$. For analytic $x\in \M$ we have
\begin{equation*}
L(x\phi^{1/2})\xi-R(x\phi^{1/2})\xi=\Lambda(x\xi)-\Lambda(x)\sigma^\phi_{-i/2}(x)=\Lambda(x\xi-\xi x).
\end{equation*}
Thus $L(a)\xi-R(a)\xi=\delta(a)$ for all $a\in\AA_\phi$. Since $\delta$ is bounded, $\xi$ is also right-bounded. In this case we can take the Tomita bimodule generated by $\H$ and $\xi$ instead of $\H$.
\end{proof}

\subsection{\texorpdfstring{The Tomita algebra $\AA_\E$}{The Tomita algebra AE}}

As announced before, we now turn to the domain of definition of the derivation associated with the quantum Dirichlet form $\E$.

\begin{theorem}\label{thm:Tomita_algebra_QMS}
Let $\M$ be a von Neumann algebra and $\phi$ a normal semi-finite faithful weight on $\M$. If $\E$ is a GNS-symmetric quantum Dirichlet form on $L_2(\M,\phi)$, then
\begin{equation*}
\AA_\E=\{a\in\AA_\phi\mid U_z a\in D(\E)\text{ for all }z\in\IC\}
\end{equation*}
is a Tomita subalgebra of $\AA_\phi$ and
\begin{equation*}
\E(ab)^{1/2}\leq \norm{\pi_l(a)}\E(b)^{1/2}+\E(a)^{1/2}\norm{\pi_r(b)}
\end{equation*}
for all $a,b\in \AA_\E$.

 For every $a\in D(\E)\cap \Lambda_\phi(\n_\phi\cap\n_\phi^\ast)$ there exists a sequence $(a_n)$ in $\AA_\E$ such that
\begin{itemize}
\item $a_n\to a$, $a_n^\sharp\to a^\sharp$ in $(D(\E),\norm\cdot_\E)$,
\item $\norm{\pi_l(a_n)}\leq \norm{\pi_l(a)}$,
\item $\pi_l(a_n)\to \pi_l(a)$ in the strong$^\ast$ topology.
\end{itemize}
In particular, $\AA_\E$ is a core for $\E$ and $\{x\in \n_\phi\cap\n_\phi^\ast\mid \Lambda_\phi(x)\in\AA_\E\}\cap \M_1$ is strongly$^\ast$ dense in $\M_1$.
\end{theorem}
\begin{proof}
Let $(P_t)$ denote the quantum dynamical semigroup associated with $\E$. By \Cref{prop:deriv_bdd} for every $t\geq 0$ there exists a normal Tomita bimodule $\H_t$ over $\AA_\phi$ and a symmetric derivation $\delta_t\colon \AA_\phi\to \H_t$ such that
\begin{equation*}
\norm{\delta_t(a)}_{\H_t}^2=\frac 1 t\langle a-P_t^{(2)}(a),a\rangle
\end{equation*}
for all $a\in \AA_\phi$.

Thus, if $a,b\in\AA_\E$, then
\begin{align*}
\frac 1 t\langle ab-P_t^{(2)}(ab),ab\rangle&=\norm{\delta_t(ab)}_{\H_t}^2\\
&=\norm{a\delta_t(b)+\delta_t(a)b}_{\H_t}^2\\
&\leq (\norm{\pi_l(a)}\norm{\delta_t(b)}_{\H_t}+\norm{\delta_t(a)}_{\H_t}\norm{\pi_r(b)})^2.
\end{align*}
By the spectral theorem, the left side converges to $\E(ab)$, while the right side converges to $(\norm{\pi_l(a)}\E(b)^{1/2}+\E(a)^{1/2}\norm{\pi_r(b)})^2$. Therefore $ab\in D(\E)$ and the claimed inequality from the  theorem holds. Replacing $a$ and $b$ by $U_z a$ and $U_z b$, respectively, one sees that $U_z(ab)\in D(\E)$. Thus $ab\in \AA_\E$.

Clearly, $\AA_\E$ is invariant under $J$ and $U_z$. The only property missing for a Tomita algebra is the density of $\AA_\E^2$ in $\AA_\E$. We will prove this later.

The density properties follow from a standard approximation argument (see \cite[Lemma 1.3]{Haa75} for example). For $a\in D(\E)\cap \Lambda_\phi(\n_\phi\cap\n_\phi^\ast)$ let
\begin{equation*}
a_n=\sqrt{\frac n\pi}\int_\IR e^{-nt^2} \Delta_\phi^{it}a\,dt.
\end{equation*}
Then $a_n\in\AA_\phi$ with 
\begin{equation*}
U_z a_n=\sqrt{\frac n\pi}\int_\IR e^{-n(t-z)^2} \Delta_\phi^{it}a\,dt
\end{equation*}
and $a_n\to a,a_n^\sharp\to a^\sharp$ in $L_2(\M,\phi)$, $\norm{\pi_l(a_n)}\leq\norm{\pi_l(a)}$ and $\pi_l(a_n)\to\pi_l(a)$ in the strong$^\ast$ topology. By Jensen's inequality for vector-valued functions \cite[Theorem 3.10]{Per74} and the $(U_t)$-invariance of $\E$ we get
\begin{equation*}
\E(U_z a_n)\leq \sqrt{\frac n\pi}\int_\IR e^{-n(t-\Re z)^2} e^{n(\Im z)^2}\E(a)\,dt\leq e^{n(\Im z)^2}\E(a).
\end{equation*}
Thus $U_z a_n\in D(\E)$ and $z\mapsto \E(U_z a_n)$ is locally bounded. This means $a_n\in\AA_\E$.

Since $\E(a_n)\leq \E(a)$, there exists a subsequence of $(a_n)$ that converges weakly in $(D(\E),\langle\cdot,\cdot\rangle_\E)$. As $a_n\to a$ in $L_2(\M,\phi)$, any subsequential weak limit of $(a_n)$ in $(D(\E),\langle\cdot,\cdot\rangle_\E)$ is necessarily $a$. Hence $(a_n)$ itself converges weakly to $a$. Finally, $\E$ is weakly lower semicontinuous by the Hahn--Banach theorem, so that
\begin{equation*}
\E(a)\leq\liminf_{n\to\infty}\E(a_n)\leq \limsup_{n\to\infty}\E(a_n)\leq \E(a).
\end{equation*}
This implies that $a_n\to a$ strongly in $(D(\E),\langle\cdot,\cdot\rangle)$.

That $\AA_\E$ is a core for $\E$ now follows from the fact that $\AA_\phi\cap D(\E)$ is a core for $\E$ by \Cref{lem:bounded_elements_form_core} or from an application of Lemma \ref{lem:approx}. The last density statement follows since $\{x\in \n_\phi\cap\n_\phi^\ast\mid\Lambda_\phi(x)\in D(\E)\}\cap \M_1$ is strongly$^\ast$ dense in $\M_1$ by \Cref{lem:bounded_elements_strong_dense}.

In particular, there exists a sequence $a_n\in \AA_\E$ such that $\pi_l(a_n)\to 1$ strongly. Thus if $b\in\AA_\E$, then $a_n b=\pi_l(a_n)b\to b$. Hence $\AA_\E^2$ is dense in $\AA_\E$, which completes the proof that $\AA_\E$ is a Tomita algebra.
\end{proof}

\begin{definition}
Let $\M$ be a von Neumann algebra and $\phi$ a normal semi-finite faithful weight on $\M$. If $\E$ is a GNS-symmetric quantum Dirichlet form on $L_2(\M,\phi)$, then
\begin{equation*}
\AA_\E=\{a\in\AA_\phi\mid U_za \in D(\E)\text{ for all }z\in\IC\}
\end{equation*}
is called the \emph{Tomita algebra associated with $\E$}.
\end{definition}

In the definition of $\AA_\E$, the map $t\mapsto U_t a$ is only assumed to have an extension to $\IC$ that is analytic with respect to the $L_2$ norm. Since $U_t a\in D(\E)$ for all $t\in \IR$, one could also require that the extension is analytic with respect to the form norm $\norm\cdot_\E$. As it turns out, this is automatically true for $a\in\AA_\E$. We prove this in two steps.

\begin{lemma}\label{lem:res_generator_group}
Let $H$ be a Hilbert space, $A$, $B$ strongly commuting positive self-adjoint operators on $H$ with $B$ non-singular. Let $\langle \xi,\eta\rangle_A=\langle \xi,\eta\rangle+\langle A\xi,A\eta\rangle$ for $\xi,\eta\in D(A)$. The restriction of $(B^{it})$ to $D(A)$ is a strongly continuous unitary group with respect to $\langle\cdot,\cdot\rangle_A$, the operator $\tilde B$ that is the restriction of $B$ to $D(\tilde B)=\{\xi\in D(A)\cap D(B)\mid B\xi\in D(A)\}$ is non-singular, positive and self-adjoint with respect to $\langle \cdot,\cdot\rangle_A$, and $B^{it}|_{D(A)}=\tilde B^{it}$ for all $t\in \IR$.
\end{lemma}
\begin{proof}
Since $A$ and $B$ are strongly commuting self-adjoint operators, we can assume by the spectral theorem that $A$ is the operator of multiplication by $f$ and $B$ is the operator of multiplication by $g$ on $L_2(X,\mu)$ for some measure space $(X,\mathcal F,\mu)$ and measurable functions $f,g\colon X\to [0,\infty)$ with $g>0$ a.e. Then $D(A)=L_2(X,(1+f^2)\mu)$, $\langle\cdot,\cdot\rangle_A$ is the $L_2$-inner product on $D(A)$, and $\tilde B$ is the operator of multiplication by $g$ on $L_2(X,(1+f^2)\mu)$. From this, all claims follow easily.
\end{proof}

\begin{lemma}\label{lem:analytic_energy}
If $\E$ is a GNS-symmetric quantum Dirichlet form on $L_2(\M,\phi)$ and $a\in\AA_\E$, then $z\mapsto \E(b,U_z a)$ is entire analytic for all $b\in D(\E)$.
\end{lemma}
\begin{proof}
Since $\E$ is a GNS-symmetric quantum Dirichlet form, the operators $A=\L_{(2)}^{1/2}$ and $B=\Delta_\phi$ commute strongly. Let $\tilde B$ be as in \Cref{lem:res_generator_group}. By \cite[Theorem VI.2.2]{Tak03}, if $a\in \AA_\E$, then $a\in D(\Delta_\phi^n)$ and $\Delta_\phi^n a=U_{-in}a\in D(\E)$ for all $n\in\IZ$. Thus $a\in\bigcap_{n\in\IZ}D(\tilde B^n)$. By \cite[Lemma VI.2.3]{Tak03}, the map $t\mapsto \tilde B^{it}a$ has an entire analytic continuation with values in $(D(\E),\langle\cdot,\cdot\rangle_\E)$. Since $\tilde B^{it}a=\Delta_\phi^{it}a$ for $t\in\IR$ by \Cref{lem:res_generator_group}, the claim follows from the uniqueness of analytic continuations.
\end{proof}

\subsection{The general case}

In this subsection we prove the existence and uniqueness of the derivation associated with a GNS-symmetric quantum Markov semigroup in the general case. Before we do so, we need the following strengthening of \Cref{prop:char_Dir_form_conservative}.

\begin{lemma}\label{lem:approx_conservative_A_E}
Let $(P_t)$ be a GNS-symmetric quantum Markov semigroup and $\E$ the associated quantum Dirichlet form. There exists a sequence $(a_n)$ in $\AA_\E$ such that $a_n^\flat=a_n$, $\norm{\pi_r(a_n)}\leq 1$, $\pi_r(a_n)\to 1$ strongly and
\begin{equation*}
\E(a_n,b^\sharp c)\to 0
\end{equation*}
for all $b,c\in\AA_\E$.
\end{lemma}
\begin{proof}
By \Cref{thm:Tomita_algebra_QMS} there exists a sequence $(x_n)$ of self-adjoint elements from $\n_\phi$ such that $\Lambda_\phi^\prime(x_n)\in \AA_\E$, $\norm{x_n}\leq 1$ and $x_n\to 1$ strongly$^\ast$. Since $b^\sharp c\in \AA_\E\cap \Lambda_\phi(\m_\phi)$ by \Cref{thm:Tomita_algebra_QMS}, the rest of the proof proceeds exactly as in \Cref{prop:char_Dir_form_conservative}.
\end{proof}

\begin{theorem}\label{thm:from_QMS_to_deriv}
Let $\M$ be a von Neumann algebra and $\phi$ a normal semi-finite faithful weight on $\M$. If $\E$ is a $(U_t)$-invariant quantum Dirichlet form on $L_2(\M,\phi)$, then there exists a Tomita bimodule $\H$ over $\AA_\E$ and a symmetric derivation $\delta\colon \AA_\E\to \H$ such that
\begin{equation*}
\langle \delta(a),\delta(b)\rangle_\H=\E(a,b)
\end{equation*}
for all $a,b\in \AA_\E$. If $\phi$ is finite, then $\delta$ can be chosen nearly inner.
\end{theorem}
\begin{proof}
\emph{Step 1:} Construction of the auxiliary Hilbert space $\K$:

On $\AA_\E\odot \AA_\E$ define a sesquilinear form by
\begin{equation*}
\langle a\otimes b,c\otimes d\rangle_\K=\frac 1 2(\E(a,cdb^\flat)+\E(abd^\flat,c)-\E(bd^\flat,a^\sharp c)).
\end{equation*}
For $t\geq 0$ let $\E_t(a)=\frac 1 t \langle a,a-P_t^{(2)}(a)\rangle$. It follows from \Cref{prop:deriv_bdd} that there exists a normal Tomita bimodule $\K_t$ and a symmetric derivation $\delta_t\colon\AA_\E\to \K$ such that $\E_t(a)=\norm{\delta_t(a)}^2$. Since $\E_t(a,b)\to \E(a,b)$ as $t\searrow 0$ for all $a,b\in\AA_\E$, we deduce from \Cref{lem:inner_prod_Tomita_bimodule} that
\begin{align*}
\langle \delta_t(a)b,\delta_t(c)d\rangle&=\frac 1 2(\E_t(a,cdb^\flat)+\E_t(abd^\flat,c)-\E_t(bd^\flat,a^\sharp c))\\
&\to \frac 1 2(\E(a,cdb^\flat)+\E(abd^\flat,c)-\E(bd^\flat,a^\sharp c))\\
&=\langle a\otimes b,c\otimes d\rangle_\K
\end{align*}
as $t\searrow 0$ for all $a,b,c,d,\in \AA_\E$. In particular,
\begin{equation*}
\left\langle\sum_j a_j\otimes b_j,\sum_k a_k\otimes b_k\right\rangle_\K=\lim_{t\to 0}\left\lVert\sum_j \delta_t(a_j)b_j\right\rVert_{\K_t}^2\geq 0
\end{equation*}
for all $a_j,b_j\in\AA_\E$.

Thus $\langle\cdot,\cdot\rangle_\K$ is a semi inner product on $\AA_\E\odot\AA_\E$. Let $\K$ denote the Hilbert space obtained from $\AA_\E\odot\AA_\E$ after separation and completion with respect to this semi inner product and write $a\otimes_\E b$ for the image of $a\otimes b$ in $\K$.

\emph{Step 2:} For all $a\in\AA_\E$ the map $b\otimes_\E c\mapsto ab\otimes_\E c-a\otimes_\E bc$ extends to a bounded linear operator $L(a)$ on $\K$ with $\norm{L(a)}\leq \norm{\pi_l(a)}$:

For $b_j,c_j\in\AA_\E$ we have
\begin{align*}
\left\lVert \sum_j ab_j\otimes_\E c_j-a\otimes_\E b_j c_j\right\rVert_\K^2&=\lim_{t\searrow 0}\left\lVert \sum_j \delta_t(ab_j)c_j-\delta_t(a)b_j c_j\right\rVert^2\\
&=\lim_{t\searrow 0}\left\lVert\sum_j a\delta_t(b_j)c_j\right\rVert^2\\
&\leq \norm{\pi_l(a)}^2\lim_{t\searrow 0}\left\lVert\sum_j \delta_t(b_j)c_j\right\rVert^2\\
&=\norm{\pi_l(a)}^2\left\lVert \sum_j a_j\otimes_\E b_j\right\rVert_\K^2.
\end{align*}
Similarly one proves that for all $a\in\AA_\E$ the map $b\otimes_\E c\mapsto b\otimes_\E ca$ extends to a bounded linear operator $R(a)$ on $\K$ with $\norm{R(a)}\leq\norm{\pi_r(a)}$. It is then not hard to check that $L(a)$ and $R(a)$ are adjointable operators and $L$, $R$ are commuting $\ast$-representations of $(\AA,^\sharp)$ and $(\AA^\op,^\flat)$, respectively.

\emph{Step 3:} For every $a\in\AA_\E$ the linear map $\psi_a\colon \K\to\IC$ determined by 
\begin{equation*}
\psi_a(b\otimes_\E c)=\frac 1 2 \E(a,bc)+\frac 1 2\E(ac^\flat,b)-\frac 1 2\E(c^\flat,a^\sharp b)
\end{equation*}
is bounded with $\norm{\psi_a}\leq \E(a)^{1/2}$:

Again one can argue by approximation and get
\begin{equation*}
\left\lvert \psi_a\left(\sum_j b_j\otimes_\E c_j\right)\right\rvert=\lim_{t\to 0}\left\lvert\left\langle \delta_t(a),\sum_j \delta_t(b_j)c_j\right\rangle\right\rvert\leq \E(a)^{1/2}\left\lVert \sum_j b_j\otimes_\E c_j\right\rVert_\K.
\end{equation*}
Let $\delta(a)$ denote the unique vector in $\K$ such that $\psi_a=\langle \delta(a),\cdot\,\rangle$.

\emph{Step 4:} For all $a,b\in \AA_\E$ we have $R(b)\delta(a)=a\otimes_\E b$ and $R$ is a non-degenerate representation:

For $c,d\in\AA_\E$ we have
\begin{align*}
\langle R(b)\delta(a),c\otimes_\E d\rangle&=\langle \delta(a),c\otimes db^\flat\rangle\\
&=\frac 1 2(\E(a,cdb^\flat)+\E(abd^\flat,c)-\E(bd^\flat,a^\sharp b))\\
&=\langle a\otimes_\E b,c\otimes_\E d\rangle.
\end{align*}
Thus $R(b)\delta(a)=a\otimes_\E b$. Since the linear span of elements of the form $a\otimes_\E b$ with $a,b\in\AA_\E$ is dense in $\K$ by definition, $R$ is non-generate.

\emph{Step 5:} For all $a,b\in\AA_\E$ we have $\delta(ab)=L(a)\delta(b)+R(b)\delta(a)$:

By Step 4, if $c\in\AA_\E$, then
\begin{align*}
R(c)\delta(ab)&=ab\otimes_\E c\\
&=L(a)(b\otimes_\E c)+R(c)(a\otimes_\E b)\\
&=L(a)R(c)\delta(b)+R(c)R(b)\delta(a)\\
&=R(c)(L(a)\delta(b)+R(b)\delta(a)).
\end{align*}
Since $R$ is non-degenerate again by Step 4, we conclude $\delta(ab)=L(a)\delta(b)+R(b)\delta(a)$.

\emph{Step 6:} For all $a\in\AA_\E$ we have $\norm{\delta(a)}_\K^2=\E(a)$:

We first prove the identity for $a\in\AA_\E\cap D(\L^{(2)})$. By \Cref{lem:approx_conservative_A_E} there exists a sequence $(b_n)$  in $\AA_\E$ such that $b_n=b_n^\flat$, $\norm{\pi_r(b_n)}\leq 1$, $\pi_r(b_n)\to 1$ strongly, and $\E(b_n,c^\sharp d)\to 0$ for all $c,d\in\AA_\E$. Moreover,
\begin{equation*}
\E(a,ab_n)=\langle \L^{(2)}(a),ab_n\rangle=\langle \L^{(2)}(a),\pi_r(b_n)a\rangle\to \E(a).
\end{equation*}
Combining this with Step 4, we get
\begin{align*}
\norm{\delta(a)}_\K^2&\geq  \langle \delta(a),\delta(a)b_n\rangle\\
&=\psi_a(a\otimes_\E b_n)\\
&=\frac 1 2\lvert\E(a,ab_n)+ \E(ab_n,a)-\E(b_n,a^\sharp a)\rvert\\
&\to \E(a).
\end{align*}
The converse inequality was already shown in Step 3.

For arbitrary $a\in\AA_\E$ we have $P_t a\in \AA_\E\cap D(\L^{(2)})$ and $\E(P_t a-a)\to 0$. By Step 3,
\begin{equation*}
\norm{\delta(a-P_t a)}_\K\leq \E(a-P_t a)^{1/2}\to 0.
\end{equation*}
This establishes $\norm{\delta(a)}_\K=\E(a)^{1/2}$ in the general case.

\emph{Step 7:} For all $a,b,c\in \AA_\E$ we have $\langle L(a)\delta(b),\delta(c)\rangle=\langle \delta(Jc),R(Ja)\delta(Jb)\rangle$:

By Step 5 we have $L(a)\delta(b)=\delta(ab)-R(b)\delta(a)$. For the first summand we can use Step 6 to get $\langle \delta(ab),\delta(c)\rangle=\E(ab,c)$. For the second summand, the result from Step 4 and the definition of $\delta$ give
\begin{equation*}
\langle R(b)\delta(a),\delta(c)\rangle=\frac 12(\E(a,cb^\flat)+\E(ab,c)-\E(b,a^\sharp c)).
\end{equation*}
Thus
\begin{align*}
\langle L(a)\delta(b),\delta(c)\rangle&=\E(ab,c)-\frac 1 2(\E(a,cb^\flat)+\E(ab,c)-\E(b,a^\sharp c))\\
&=\frac 1 2\E(J c,Jb\cdot Ja)+\frac 1 2\E(Jc(Ja)^\flat,Jb)-\frac 1 2\E(U_{-i/2}a,U_{i/2}(c b^\flat))\\
&=\frac 1 2\E(J c,Jb\cdot Ja)+\frac 1 2\E(Jc(Ja)^\flat,Jb)-\frac 1 2 \E((Ja)^\flat,(Jc)^\sharp Jb)\\
&=\langle \delta(Jc),R(Ja)\delta(Jb)\rangle,
\end{align*}
where we used $\E\circ J=\E$ and the invariance of $\E$ under $(U_t)$ in the second line.

It follows that $L(a)\delta(b)\mapsto R(Ja)\delta(Jb)$ is well-defined and extends to an isometric anti-linear operator $\J$ on $\K$. It is easy to see that $\J$ is an anti-unitary involution and $\J L(a)=R(Ja)\J$ for all $a\in\AA_\E$.

\emph{Step 8:} The $\ast$-representation $L$ of $\AA_\E$ on $\K$ is non-degenerate:

By the previous step, $L(a)=\J R(Ja) \J$ for all $a\in\AA_\E$. Since $R$ is non-degenerate by Step 4, it follows that $L$ is non-degenerate, too.

\emph{Step 9:} For all $t\in \IR$ the map $a\otimes_\E b\mapsto U_t a\otimes_\E U_t b$ extends to a unitary operator $\U_t$ on $\K$ such that $\U_t(L(a)R(b)\xi)=L(U_t a)R(U_t b)\U_t \xi$ for all $a,b\in \AA_\E$, $\xi\in \K$ and $\U_t\J=\J\U_t$. Moreover, the family $(\U_t)$ is a strongly continuous unitary group:

For $a,b,c,d\in\AA_\E$ we have
\begin{align*}
&\quad\;\langle U_t a\otimes_\E U_t b,U_t c\otimes_\E U_t d\rangle\\
&=\frac 1 2(\E(U_t a,U_t(cdb^\flat))+\E(U_t(abd^\flat),U_t c)-\E(U_t(bd^\flat),U_t(a^\flat c)))\\
&=\frac 1 2((\E(a,cdb^\flat)+\E(abd^\flat,c)-\E(bd^\flat,a^\flat c))\\
&=\langle a\otimes_\E b,c\otimes_\E d\rangle,
\end{align*}
where we used that $U_t$ is an algebra homomorphism on $\AA_\E$ that commutes with $J$ and $\E$ is invariant under $U_t$. From this it is easy to see that $(\U_t)$ is a unitary group.

To show that $(\U_t)$ is strongly continuous, it suffices to show that it is weakly continuous by standard semigroup theory. Moreover, as $(U_t)$ is uniformly bounded, it is enough to check that $t\mapsto \langle a\otimes_\E b,\U_t(c\otimes_\E d)\rangle$ is continuous for all $a,b,c,d\in\AA_\E$. This follows from the fact that $(U_t)$ is a strongly continuous unitary group on $(D(\E),\langle\cdot,\cdot\rangle_\E)$ by \Cref{lem:analytic_energy} and multiplication by elements from $\AA_\E$ is continuous with respect to the form norm $\norm\cdot_\E$ by \Cref{thm:Tomita_algebra_QMS}.

The remaining properties of $\U_t$ follow by direct computation.

\emph{Step 10:} $\U_t(\delta(a))=\delta(U_t a)$ for all $a\in\AA_\E$, $t\in\IR$:

For $b,c\in\AA_\E$ we have
\begin{align*}
\langle \U_t(\delta(a)),b\otimes_\E c\rangle&=\langle \delta(a),U_{-t} a\otimes_\E U_{-t}b\rangle\\
&=\frac 1 2\E(a,U_{-t}(bc))+\frac 1 2\E(a(U_{-t} c)^\flat,U_{-t}b)-\frac 1 2\E((U_{-t}c)^\flat,a^\sharp U_{-t}b)\\
&=\frac 1 2\E(U_t a,bc)+\frac 1 2\E((U_t a)c^\sharp,b)-\frac 1 2\E(c^\flat,(U_t a)^\sharp b)\\
&=\langle \delta(U_t a),b\otimes_\E c\rangle,
\end{align*}
where we used that $\E$ is invariant under $U_t$ in the third line. Since the linear span of elements of the form $b\otimes_\E c$ with $b,c\in\AA_\E$ is dense in $\K$, the claim follows.

\emph{Step 11:} For all $a\in \AA_\E$ the map $z\mapsto \delta(U_z a)$ is entire analytic:

By elementary properties of Banach-valued analytic functions, it suffices to show that $z\mapsto \langle \xi,\delta(U_z a)\rangle$ is analytic for all $\xi$ in the closure of $\delta(\AA_\E)$. By \Cref{lem:analytic_energy} the map $z\mapsto U_z a$ is entire analytic with respect to the form norm $\norm{\cdot}_\E$. In particular, since $\norm{\delta(U_z a)}=\E(U_z a)^{1/2}$ by Step 6, the map $z\mapsto \delta(U_z a)$ is locally bounded. Hence it is enough to show that $z\mapsto \langle\xi,\delta(U_z a)\rangle$ is analytic for all $\xi\in \delta(\AA_\E)$, which follows from \Cref{lem:analytic_energy} combined with Step 6.

\emph{Conclusion:} Let $\H$ be the space of all elements $\xi\in \K$ for which $t\mapsto \U_t \xi$ has an entire analytic continuation. By Steps 10 and 11, $\delta(\AA_\E)$ is contained in $\H$. Moreover, it follows easily from Step 9 that $L(a)$ and $R(a)$ leave $\H$ invariant for all $a\in\AA_\E$, and the same holds for $\J$. By the properties established in the previous steps, this makes $\K$ into a Tomita bimodule over $\AA_\E$ and $\delta$ into a symmetric derivation such that
\begin{equation*}
\E(a,b)=\langle\delta(a),\delta(b)\rangle
\end{equation*}
for all $a,b\in\AA_\E$.

If $\phi$ is finite, then the derivations $\delta_t$ from Step 1 can be chosen inner by \Cref{prop:deriv_bdd}, so that $\delta$ is nearly inner by construction.
\end{proof}

\subsection{Uniqueness}

In general, the Tomita bimodule and derivation in \Cref{thm:from_QMS_to_deriv} are not unique because one can always artificially enlarge $\H$. To get uniqueness, one needs to assume that the Tomita bimodule is generated by the derivation in the following sense. Let $\AA$ be a Tomita algebra and $\HH$ a Tomita bimodule over $\AA$. Write $\tilde \AA$ for the unitization of $\AA$, and extend the maps $L$ and $R$ to $\tilde\AA$ in the obvious way. If $\delta\colon\AA\to \H$ is a symmetric derivation, say that $\H$ is generated by $\delta$ if
\begin{equation*}
\operatorname{lin}\{L(a)R(b)\delta(c)\mid a,b\in\tilde \AA,\,c\in\AA\}=\H.
\end{equation*}
By the product rule, this is equivalent to $\operatorname{lin}\{L(a)\delta(b)\mid a\in\tilde \AA,\,b\in\AA\}=\H$ or $\operatorname{lin}\{R(a)\delta(b)\mid a\in\tilde \AA,\,b\in\AA\}=\H$.

\begin{theorem}\label{thm:deriv_unique}
Let $\E$ be a GNS-symmetric quantum Dirichlet form. If for $j\in\{1,2\}$, $\H_j$ is a Tomita bimodule over $\AA_\E$ and $\delta_j\colon\AA_\E\to\H_j$ is a symmetric dervation such that $\H_j$ is generated by $\delta_j$ and $\E(a)=\norm{\delta_j(a)}^2$ for all $a\in\AA_\E$, then there exists a unique bijective isometric bimodule map $V\colon \H_1\to\H_2$ that intertwines the conjugation operators $\J_1$, $\J_2$ and the groups $(\U_z^{(1)})$, $(\U_z^{(2)})$ and satisfies $V\circ\delta_1=\delta_2$.
\end{theorem}
\begin{proof}
Since $\H_1$ and $\H_2$ are generated by $\delta_1$ and $\delta_2$, respectively, and $V$ is a bimodule map, it is uniquely determined by the requirement $V\circ \delta_1=\delta_2$, if it exists.

For the existence, define $V(R_1(b)\delta_1(a))=R_2(b)\delta_2(a))$ for $a\in\AA_\E$, $b\in\tilde\AA_\E$. By \Cref{lem:inner_prod_Tomita_bimodule} this map is well-defined and isometric. It follows from the product rule that $V$ is a bimodule map and $V$ is surjective because $\H_2$ is generated by $\delta_2$. That $V$ intertwines the conjugation operators $\J_1$, $\J_2$ and the group $(\U_z^{(1)})$, $(\U_z^{(2)})$ is  immediate from the definition of a symmetric derivation.
\end{proof}

If $\phi$ is finite, we have seen in \Cref{thm:from_deriv_to_QMS} that every nearly inner symmetric derivation gives rise to a GNS-symmetric quantum Dirichlet form. Conversely, every GNS-symmetric quantum Dirichlet form induces a nearly inner symmetric derivation by \Cref{thm:from_QMS_to_deriv}. In general, this correspondence is not one-to-one, as different derivations on Tomita subalgebras of $\AA_\phi$ can have the same closure. Among these derivations, the one constructed in \Cref{thm:from_QMS_to_deriv} has the maximal domain in the following sense:

We say a symmetric derivation $\delta\colon \AA\to\H$ is \emph{maximally defined} if it is closable and
\begin{equation*}
\AA=\{a\in\AA_\phi\mid \Delta_\phi^n a\in D(\bar\delta)\text{ for all }n\in\IZ\}.
\end{equation*}

The following result is then immediate from the discussion above.
\begin{corollary}\label{cor:bijection_deriv_QMS}
Let $\M$ be a von Neumann algebra and $\phi$ a normal faithful state on $\M$. There is a one-to-one correspondence between maximally defined nearly inner symmetric derivations on Tomita subalgebras of $\AA_\phi$ and GNS-symmetric quantum Dirichlet forms on $L_2(\M,\phi)$.
\end{corollary}

\section{\texorpdfstring{Carré du champ and $\Gamma$-regularity}{Carré du champ and Γ-regularity}}\label{sec:carre_du_champ}

Let $\AA$ be a Tomita algebra and $\H$ a Tomita bimodule over $\AA$. Recall that $\H$ is called a normal Tomita bimodule if the map $\pi_l(a)\mapsto L(a)$ extends to a normal representation of $\pi_l(\AA)^{\prime\prime}$. Even if $(P_t)$ is a symmetric QMS on a commutative von Neumann algebra, the Tomita bimodule from \Cref{thm:from_QMS_to_deriv} is in general not normal. For example this is a typical feature of Dirichlet forms on fractals if the reference measure is taken to be the Hausdorff measure (see \cite{Kus89} for example).

In [Wir20] we gave a characterization of quantum Dirichlet forms for which the associated Tomita bimodule is normal in the tracially symmetric case in terms of the so-called carré du champ. In this section we show that this characterization extends to GNS-symmetric quantum Dirichlet forms.

Let $\M$ be a von Neumann algebra, $\phi$ a normal semi-finite faithful weight on $\M$ and $(P_t)$ a $\phi$-symmetric QMS on $\M$ with associated quantum Dirichlet form $\E$. We use the notation from \Cref{thm:from_QMS_to_deriv} for the associated Tomita algebra and Tomita bimodule.

Let $A$ be the norm closure of $J\pi_r(\AA_\E)J$, which is strongly dense in $\M$. For $\xi,\eta\in\overline\H$ the functional
\begin{equation*}
J\pi_r(\AA_\E)J\to \IC,\,J\pi_r(a)^\ast J\mapsto \langle \xi,R(a)\eta\rangle
\end{equation*}
extends to a bounded linear functional $\hat\Gamma(\xi,\eta)$ on $A$ with $\norm{\hat\Gamma(\xi,\eta)}\leq\norm{\xi}\norm{\eta}$. If $a,b\in D(\E)$, we write $\Gamma(a,b)$ for $\hat\Gamma(\bar\delta(a),\bar\delta(b))$. Furthermore, we use the usual convention $\hat\Gamma(\xi)=\hat\Gamma(\xi,\xi)$ and the same for $\Gamma$.

\begin{remark}\label{rmk:carre_du_champ}
If $a,b,c\in \AA_\E$, then 
\begin{equation*}
\langle\Gamma(a,b),J\pi_r(c)^\ast J\rangle_{A^\ast,A}=\frac 1 2(\E(a,bc)+\E(ac^\flat,b)-\E(c^\flat,a^\sharp b))
\end{equation*}
by \Cref{lem:inner_prod_Tomita_bimodule}.
\end{remark}

\begin{theorem}\label{thm:char_normal_bimodule}
For a quantum Dirichlet form $\E$ with associated Tomita algebra $\AA_\E$, Tomita bimodule $\H$ and derivation $\delta$, the following assertions are equivalent:
\begin{enumerate}[(i)]
\item $\H$ is a normal Tomita bimodule.
\item The map $J\pi_r(\AA_\E)J\to B(\overline\H),\,J\pi_r(a)^\ast J\mapsto R(a)$ extends to a normal representation of $\M^{\op}$ on $\overline\H$.
\item For all $\xi,\eta\in \overline\H$ the map $\hat\Gamma(\xi,\eta)$ is $\sigma$-weakly continuous.
\item For every $a\in \AA_\E$ the map $\Gamma(a)$ is $\sigma$-weakly continuous.
\end{enumerate}
\end{theorem}
\begin{proof}
(i)$\iff$(ii): For $a\in \AA_\E$ let $\tilde L(\pi_l(a))=L(a)$ and $\tilde R(J\pi_r(a)^\ast J)=R(a)$. Since
\begin{equation*}
J\pi_r(Ja)^\ast J b=J(Jb\cdot (Ja)^\flat)=a^\flat b=\pi_l(a)^\ast b
\end{equation*}
for all $b\in \AA_\E$, we have $J \pi_r(\AA_\E)J=\pi_l(\AA_\E)$. Moreover,
\begin{equation*}
\J \tilde L(\pi_l(a))\xi=\J L(a)\xi=R(Ja)\J\xi=\tilde R(J\pi_r(Ja)^\ast J)\J\xi=\tilde R(\pi_l(a)^\ast)\J\xi
\end{equation*}
 for all $\xi\in\H$. Thus $\tilde R(x)=\J \tilde L(x)^\ast \J$ for all $x\in \pi_l(\AA_\E)$. From this it is easy to see that $\tilde L$ extends to a normal representation if and only if $\tilde R$ does.
 
(ii)$\implies$(iii), (iii)$\implies$(iv) are clear from the definitions.
 
(iv)$\implies$(iii): By polarization, $\Gamma(a,b)$ is $\sigma$-weakly continuous for all $a,b\in \AA_\E$. If $\xi=R(b)\delta(a)$, $\eta=R(d)\delta(c)$ with $a,b,c,d\in \AA_\E$, then
\begin{align*}
 \langle \hat\Gamma(\xi,\eta),J\pi_r(a)^\ast J\rangle_{A^\ast,A}&=\langle R(b)\delta(a),R(a)R(d)\delta(c)\rangle\\
 &=\langle \delta(a),R(b^\flat a d)\delta(c)\rangle\\
 &=\langle \Gamma(a,c),J\pi_r(b^\flat a d)^\ast J\rangle_{A^\ast,A}.
\end{align*}
Thus $\hat \Gamma(\xi,\eta)$ is $\sigma$-weakly continuous. By sesquilinearity, this extends to $\xi,\eta$ in $\H$. For arbitrary $\xi,\eta\in \overline\H$ take sequences $(\xi_k)$, $(\eta_k)$ in $\H$ such that $\xi_k\to\xi$, $\eta_k\to\eta$. Since
\begin{equation*}
\norm{\hat\Gamma(\xi,\eta)-\hat\Gamma(\xi_k,\eta_k)}\leq \norm{\xi-\xi_k}\norm{\eta_k}+\norm{\eta_k-\eta}\norm{\xi}
\end{equation*}
and the norm limit of $\sigma$-weakly continuous functionals is $\sigma$-weakly continuous, we obtain that $\hat\Gamma(\xi,\eta)$ is $\sigma$-weakly continuous.

(iii)$\implies$(ii): For $\xi,\eta\in \overline\H$ let $\omega_{\xi,\eta}$ denote the $\sigma$-weakly continuous extension of $\hat\Gamma(\xi,\eta)$ to $\M^\op$. By the Kaplansky density theorem, $\norm{\omega_{\xi,\eta}}=\norm{\hat\Gamma(\xi,\eta)}\leq\norm{\xi}\norm{\eta}$. Thus for every $x\in\M^\op$ there exists $\tilde R(x)\in B(\overline\H)$ with $\norm{\tilde R(x)}\leq\norm{x}$ such that $\omega_{\xi,\eta}(x)=\langle \xi,\tilde R(x)\eta\rangle$. Clearly $\tilde R$ is linear, weakly continuous on the unit ball and extends the map from (ii). Since the restriction of  $\tilde R$ on $J\pi_r(\AA_\E) J$ is a $\ast$-homomorphism (for the multiplication in $\M^\op$) and it is weakly continuous, it follows that $\tilde R$ is a $\ast$-homomorphism on $\M^\op$ by a density argument. As $R$ is non-degenerate, $\tilde R$ is unital. Thus $\tilde R$ is a normal representation of $\M^\op$ on $\overline\H$.
\end{proof}

\begin{remark}
According to \Cref{rmk:carre_du_champ}, if $a\in\AA_\E$, then $\Gamma(a)$ can be expressed in terms of the form $\E$ without reference to $\H$ and $\delta$. Thus (iv) in the previous theorem gives a criterion for the normality of $\H$ that can be checked without the need for an explicit construction of $\H$ and $\delta$.
\end{remark}

\begin{definition}
We say the quantum Dirichlet form $\E$ or the associated quantum Markov semigroup $(P_t)$ is \emph{$\Gamma$-regular} if one of the equivalent conditions from \Cref{thm:char_normal_bimodule} holds.
\end{definition}

\begin{remark}
$\Gamma$-regularity also has a probabilistic counterpart in terms of Markov dilations \cite{JRS}, but we will not go into details here.
\end{remark}

As a consequence of the construction from Subsection \ref{subsec:Fock_construction} we obtain that the generator of a $\Gamma$-regular GNS-symmetric quantum Dirichlet form is of the form $\delta^\ast\bar\delta$, where $\delta$ is a twisted derivation with values in a noncommutative $L_2$ space.

\begin{corollary}\label{cor:deriv_triple}
If $\E$ is a $\Gamma$-regular GNS-symmetric quantum Dirichlet form with generator $\L^{(2)}$, then there exists a von Neumann algebra $\hat \M\supset \M$, a weight $\hat \phi$ on $\hat \M$, a faithful normal conditional expectation $E\colon\hat \M\to\M$ such that $\phi\circ E=\hat\phi$, and a closable linear map $\delta\colon\AA_\E\to L_2(\hat\M,\hat\phi)$ such that
\begin{equation*}
\delta(\Lambda_\phi(xy))=x\delta(\Lambda_\phi(y))+\delta(\Lambda_\phi(x))\sigma^\phi_{i/2}(y)
\end{equation*}
for all $x,y\in \Lambda_\phi^{-1}(\AA_\E)$ and
\begin{equation*}
\L^{(2)}=\delta^\ast\bar\delta.
\end{equation*}
\end{corollary}

\begin{remark}
The twisted product rule in the previous corollary can also be expressed in terms of the embeddings $\Delta_\phi^\alpha\Lambda_\phi$ for $\alpha\in [0,\frac 1 2]$ as follows:
\begin{equation*}
\delta(\Delta_\phi^\alpha\Lambda_\phi(xy))=\sigma^\phi_{-i\alpha}(x)\delta(\Delta_\phi^\alpha\Lambda_\phi(y))+\delta(\Delta_\phi^\alpha\Lambda_\phi(x))\sigma^\phi_{i(1/2-\alpha)}(y).
\end{equation*}
\end{remark}

\begin{remark}
In the case when $\phi$ is a trace, the existence of a derivation as in Subsection \ref{subsec:Fock_construction} for $\Gamma$-regular quantum Dirichlet forms was established in \cite{JRS}. The triple $(\AA_\E,\hat \M,\delta)$ is called a \emph{derivation triple} in this setting \cite{BGJ22,LJR20}.
\end{remark}

\section{Examples}\label{sec:examples}

In this final section we give examples of quantum Dirichlet forms for which we can describe the associated Tomita bimodule and derivation more explicitly. We start by discussing how tracially symmetric quantum Markov semigroups and quantum Markov semigroups on finite-dimensional von Neumann algebras fit into the picture in the first two subsections, before we treat new examples with depolarizing semigroups, the Ornstein--Uhlenbeck semigroup on free Araki--Woods factors and translation-invariant quantum Markov semigroups on compact quantum groups in the last three subsections.

\subsection{Tracially symmetric quantum Markov semigroups}\label{subsec:trace_sym_QMS}

Let $\M$ be a semi-finite von Neumann algebra and $\tau$ a normal semi-finite faithful trace on $\M$. In this case the modular group $\sigma^\tau$ is trivial and the involutions $^\sharp$ and $^\flat$ on $\AA_\tau$ coincide with $J_\tau$.

If $\E$ is a quantum Dirichlet form on $L_2(\M,\tau)$, then it is automatically GNS-symmetric, and $\AA_\E=D(\E)\cap \Lambda_\tau(\n_\tau)$. Let $\H$ be the associated Tomita bimodule. Since 
\begin{equation*}
\U_t(\delta(a)b)=\delta(\sigma^\tau_t(a))\sigma^\tau_t(b)=\delta(a)b
\end{equation*}
and elements of the form $\delta(a)b$ with $a,b\in\AA_\E$ linearly span $\H$, the group $(\U_t)$ is trivial.

In other words, we have a Hilbert bimodule $\overline\H$ over $\AA_\E$ with contractive and non-degenerate left and right action, an anti-unitary involution $\J$ on $\overline\H$ that intertwines the left and right action and a derivation $\delta\colon \AA_\E\to \overline \H$ such that $\J\circ\delta=\delta\circ J$ and
\begin{equation*}
\E(a,b)=\langle\delta(a),\delta(b)\rangle
\end{equation*}
for all $a,b\in\AA_\E$. This is the same data as one gets from the construction in \cite{CS03} for tracially symmetric quantum Markov semigroups.

\subsection{QMS on finite-dimensional von Neumann algebras}\label{subsec:fin_dim_QMS}

Let $\M\subset M_n(\IC)$ be a unital $^\ast$-subalgebra, necessarily weakly closed, and $h\in \M_+$ with $\tr(h)=1$. Write $\phi$ for the state induced by $h$, that is, $\phi(x)=\tr(xh)$. We assume that $\phi$ is faithful, or, equivalently, that $h$ is invertible.

By Alicki's theorem (see \cite[Theorem 3]{Ali76}, \cite[Theorem 3.1]{CM17} for the case of matrix algebras and \cite[Corollary 5.4]{Wir22} for the extension to arbitrary finite-dimensional von Neumann algebras) a continuous semigroup $(P_t)$ on $\M$ is a $\phi$-symmetric QMS if and only if there exist $\omega_1,\dots,\omega_m\in \IR$ and $v_1,\dots,v_m\in M_n(\IC)$ satisfying
\begin{itemize}
\item $\tr(v_j)=0$ for $1\leq j\leq m$,
\item $\tr(v_j^\ast v_k)=\delta_{jk}\tr(v_j^\ast v_j)$ for $1\leq j,k\leq m$,
\item $\{v_j\mid 1\leq j\leq m\}=\{v_j^\ast\mid 1\leq j\leq m\}$,
\item $h v_j h^{-1}=e^{-\omega_j}v_j$ for $1\leq j\leq m$
\end{itemize}
such that the generator $\L$ of $(P_t)$ is given by
\begin{equation*}
\L=\sum_{j=1}^m (e^{-\omega_j/2}v_j^\ast[v_j,\cdot\,]+e^{\omega_j/2}[\,\cdot\,,v_j]v_j^\ast).
\end{equation*}
Note that in general the matrices $v_j$ cannot be chosen in $\M$.

In this case, the differential structure from \Cref{thm:from_QMS_to_deriv} coincides with the differential structure studied in [CM17,Wir22]. Let us first fix some notation. The space $L_2(\M,\phi)$ can be identified with $\M$ with the inner product $\langle x,y\rangle_h=\tr(x^\ast y h)$. Clearly, this inner product can be extended to $M_n(\IC)$, and we write $H$ for the resulting Hilbert space.

With these identifications, the Tomita algebra $\AA_\phi$ is $L_2(\M,\phi)$ with the product from $\M$, the involution $a^\sharp=a^\ast$ and $U_z a=h^{iz}a h^{-iz}$. Since the quantum Dirichlet form $\E$ associated with $(P_t)$ is defined everywhere, we have $\AA_\E=\AA_\phi$.

Let $\H$ be the linear span of $\{([v_j,a]b)_j\mid a,b\in L_2(\M,\phi)\}$ inside $H^{\oplus m}$. If we endow $H^{\oplus m}$ with the componentwise left and right actions of $L_2(\M,\phi)$, then $\H$ is a sub-bimodule: If $a,b,c\in L_2(\M,\phi)$, then $([v_j,a]bc)_j\in \H$ and
\begin{equation*}
(a[v_j,b]c)_j=([v_j,ab]c-[v_j,a]bc)_j\in\H.
\end{equation*}
The space $\H$ becomes a Tomita bimodule over $L_2(\M,\phi)$ with
\begin{align*}
\U_z\xi&=(e^{i\omega_j z}h^{iz}\xi_j h^{-iz})_j,\\
\J([v_j,a]b)_j&=(h^{1/2}(b^\ast[v_{j^\ast},a]^\ast)h^{1/2})_j,
\end{align*}
where $j^\ast\in\{1,\dots,m\}$ is the unique index such that $v_j^\ast=v_{j^\ast}$.

The map $\delta\colon L_2(\M,\phi)\to\H$ given by $\delta(a)=(i e^{-\omega_j/4}[v_j,a])_j$ is a symmetric derivation and
\begin{align*}
\langle \delta(a),\delta(b)\rangle&=\sum_{j=1}^m \tr((i e^{-\omega_j/4}[v_j,a])^\ast (i e^{-\omega_j/4}[v_j,b])h)\\
&=\sum_{j=1}^m e^{-\omega_j/2}\tr(a^\ast v_j^\ast [v_j,b]h-a^\ast [v_j,b]\sigma v_j^\ast)\\
&=\sum_{j=1}^m \tr(e^{-\omega_j/2}a^\ast v_j^\ast[v_j,b]h-e^{\omega_j/2}a^\ast[v_j,b]v_j^\ast h)\\
&=\langle a,\L(b)\rangle.
\end{align*}

\subsection{Depolarizing semigroups}\label{subsec:depolarizing_QMS}

Let $\M$ be a von Neumann algebra and $\phi$ a normal faithful state on $\M$. The depolarizing semigroup (with respect to $\phi$) is given by $P_t(x)=e^{-t}x+(1-e^{-t})\phi(x)$.  It is evident that $(P_t)$ is a GNS-symmetric quantum Markov semigroup on $\M$. 

The associated Tomita algebra is simply $\AA_\phi$ and the associated Tomita bimodule is a subbimodule of $\AA_\phi\odot \AA_\phi$ with the usual inner product, $L=\pi_l\otimes \pi_l$, $R=\pi_r\otimes \pi_r$, $\U_z=U_z\otimes U_z$ and $\J(a\otimes b)=J_\phi b\otimes J_\phi a$. The associated derivation is then given by $\delta(a)=i(a\otimes\Omega_\phi-\Omega_\phi\otimes a)$, where $\Omega_\phi$ denotes the cyclic vector in $L_2(\M,\phi)$ representing $\phi$.

\subsection{Free Araki--Woods factors}\label{subsec:Araki-Woods}

Free Araki--Woods factors were introduced by Shlyakhtenko in the context of free probability \cite{Shl97}. We give a slightly modified, but equivalent definition here, which stresses the analogy to the definition of algebras $\Gamma_\M(\H)$ from Subsection \ref{subsec:Fock_construction} (see \cite{Vae05} for example). Let $H$ be a Hilbert space, $I\colon H\to H$ an anti-unitary involution and $(V_t)$ a strongly continuous unitary group on $H$ such that $V_t I=I V_t$ for all $t\in\IR$. Let $A$ be the unique positive self-adjoint non-singular operator on $H$ such that $V_t=A^{it}$ for all $t\in \IR$, and let $T=IA^{-1/2}$.

On the full Fock space
\begin{equation*}
\F(H)=\IC\oplus \bigoplus_{n=1}^\infty H^{\otimes n}
\end{equation*}
define the creation and annihilation operators $a(\xi)$ and $a^\ast(\xi)$ for $\xi\in H$ as in Subsection \ref{subsec:Fock_construction}. Let $s(\xi)=a(\xi)+a^\ast(\xi)$. The free Araki--Woods factor $\Gamma(H,(V_t))^{\prime\prime}$ is the von Neumann algebra generated by the operators $s(\xi)$ for $\xi\in D(T)$ with $T\xi=\xi$.

If we view $H$ as a Tomita $\IC$-$\IC$-correspondence with conjugation $I$ and unitary group $(V_{-t})$, then $\Gamma(H,(V_t))^{\prime\prime}$ is the same as the free Gaussian algebra $\Gamma_\IC(H)$ from Subsection \ref{subsec:Fock_construction}. The resulting von Neumann algebras can be viewed as operator-valued versions of Shlyakhtenko's free Araki--Woods factors \cite{Shl97}, which are a special case of our construction when $\AA=\IC$, so that they might be of independent interest in operator-valued free probability theory. In this article, the construction is used to extend the construction of derivation triples from \cite{JRS} to GNS-symmetric quantum Markov semigroups.

Write $\Omega$ for the vector $(1,0,\dots)\in\F(H)$ and $\phi$ for the associated vector state on $\Gamma(H,(V_t))^{\prime\prime}$. The state $\phi$ is called the free quasi-free state. The vector $\Omega$ is cyclic and separating for $\Gamma(H,(V_t))^{\prime\prime}$, which means that the GNS Hilbert space associated with $\phi$ is $\F(H)$ with the defining action of $\Gamma(H,(V_t))^{\prime\prime}$. The associated modular group acts as $\Delta_\phi^{it}|_{H^{\otimes n}}=(V_{-t})^{\otimes n}$ and the associated modular conjugation acts as $J_\phi|_{H^{\otimes n}}=I^{\otimes n}\tau_n$, where $\tau_n$ is the reversal map given by $\tau_n(\xi_1\otimes\dots\otimes\xi_n)=\xi_n\otimes\dots\otimes\xi_1$.

The number operator $N$ on $\F(H)$ is the positive self-adjoint operator on $\F(H)$ with domain 
\begin{equation*}
D(N)=\{\xi\in \F(H)\mid \sum_{n=0}^\infty n^2\norm{\xi_n}^2<\infty\}
\end{equation*}
that acts on $H^{\otimes n}$ as multiplication by $n$. It generates a symmetric QMS on $\F(H)$, which clearly commutes with $(\Delta_\phi^{it})$. Let $(P_t)$ denote the associated GNS-symmetric QMS on $\Gamma(H,(V_t))^{\prime\prime}$ and $\E$ the associated quantum Dirichlet form.

The QMS $(P_t)$ can equivalently be described as $\Gamma^{\prime\prime}(e^{-t}\mathrm{id}_H)$, where $\Gamma^{\prime\prime}$ is Shlyakhtenko's second quantization functor (see also \cite{Hia03} or \cite{BM21}, where this QMS is called the Ornstein--Uhlenbeck semigroup on $\Gamma(H,(V_t))^{\prime\prime}$).

Let $\AA=\operatorname{lin}\bigcup_n D(T^{\otimes n})$. Inductively one can show that every $\eta\in \AA$ there exists $x$ in the algebra generated by $\{s(\xi)\mid \xi\in D(T),\,T\xi=\xi\}$ such that $x\Omega=\eta$, and $x$ is unique since $\Omega$ is separating for $\Gamma(H,(V_t))$. This is the Wick word decomposition of $\eta$ and denoted by $x=W(\eta)$. In particular, $\AA$ is a Tomita subalgebra of $\AA_\E$ and the multiplication is given by $\xi\eta=W(\xi)W(\eta)\Omega$. Moreover, $\AA$ is a core for $\E$ by \Cref{lem:approx}. For this reason we can restrict our attention to $\AA$ instead of $\AA_\E$.

Let us first describe the Tomita bimodule $\H$. On $H\oplus H$ we have the anti-linear involution $I\oplus I$ and the strongly continuous unitary group $(V_t\oplus V_t)$, which commutes with $I\oplus I$. Thus we can construct $\Gamma(H\oplus H,(V_t\oplus V_t))^{\prime\prime}$ as before, and the GNS Hilbert space with respect to the free quasi-free state $\psi$ on $\Gamma(H\oplus H,(V_t\oplus V_t))^{\prime\prime}$ is canonically identified with $\F(H\oplus H)$. By a slight abuse of notation, we also write $\Omega$ for the vacuum vector in $\F(H\oplus H)$.

In particular, the group generated by the modular operator and modular conjugation operator act  as $\Delta_\psi^{it}|_{(H\oplus H)^{\otimes n}}=(V_{-t}\oplus V_{-t})^{\otimes n}$ and $J_\psi|_{(H\oplus H)^{\otimes n}}=(I\oplus I)^{\otimes n}\Sigma_n$, where $\tau_n$ is the reversal map as defined above. Let $\H=\operatorname{lin}\bigcup_n D((T\oplus T)^{\otimes n})$, which is contained in $\AA_\psi$ under these identifications.

As another application of the second quantization theorem, the map that sends $s(\xi)\mapsto s((\xi,0))$ for $\xi\in D(T)$ extends to a normal unital state-preserving $\ast$-homomorphism $\Phi$ from $\Gamma(H,(V_t))^{\prime\prime}$ to $\Gamma(H\oplus H,(V_t\oplus V_t))^{\prime\prime}$. The defining property of $\Phi$ implies that $\Phi(W(\xi))=W(\xi,0)$ for $\xi\in \AA$.

Define left and right actions of $\AA$ on $\F(H\oplus H)$ by
\begin{align*}
L(\xi)x\Omega&=\Phi(W(\xi))x\Omega\\
R(\xi)x\Omega&=x\Phi(W(\xi))\Omega
\end{align*}
for $\xi\in \AA$ and $x\in \Gamma(H\oplus H,(V_t\oplus V_t))^{\prime\prime}$.

Clearly, $L(\xi)$ and $R(\xi)$ map $\H$ into $\H$ for $x\Omega\in\AA$. Together with the restrictions of $(\Delta_\psi^{it})$ and $J_\psi$, this makes $\H$ into a Tomita bimodule over $\AA$.

Now let us describe the derivation $\delta$. Let $\Sigma\colon H\oplus H\to H\oplus H$ be the flip map, that is, $\Sigma(\xi,\eta)=(\eta,\xi)$. The family $(\alpha_t)_t=(\F(e^{it\Sigma})\,\cdot\,\F(e^{-it\Sigma}))_t$ is a weak$^\ast$ continuous group of $\ast$-automorphisms on $B(\F(H\oplus H))$, and by the second quantization lemma, it maps $\Gamma(H\oplus H,(V_t\oplus V_t))^{\prime\prime}$ into itself.

Let $D(\partial)$ be the set of all $x\in \Gamma(H\oplus H,(V_t\oplus V_t))^{\prime\prime}$ for which $\lim_{t\to 0}\frac 1 t(\alpha_t(x)-x)$ exists in the strong$^\ast$-topology. Since $(\alpha_t)$ consists of $\ast$-automorphisms, $D(\partial)$ is a $\ast$-subalgebra of $\Gamma(H\oplus H,(V_t\oplus V_t))^{\prime\prime}$ and the operator
\begin{equation*}
\partial\colon D(\partial)\to \Gamma(H\oplus H,(V_t\oplus V_t))^{\prime\prime},\,x\mapsto\frac 1 i\lim_{t\searrow 0}\frac 1 t(\alpha_t(x)-x)
\end{equation*}
is a $\ast$-derivation. Moreover, since $\Sigma$ commutes with $\Delta_\psi^{it}$, the derivation $\partial$ commutes with $\sigma^\psi$.

Let 
\begin{equation*}
\delta\colon \AA\to\H,\,\xi\mapsto \partial(\Phi(W(\xi)))\Omega=\partial(W(\xi,0))\Omega.
\end{equation*}
We have
\begin{align*}
\delta(\xi\eta)&=\partial(\Phi(W(\xi)W(\eta)))\Omega\\
&=\Phi(W(\xi))\partial(\Phi(W(\eta)))\Omega+\partial(\Phi(W(\xi)))\Phi(W(\eta))\Omega\\
&=L(\xi)\delta(\eta)+R(\eta)\delta(\xi).
\end{align*}
Since $\partial$ commutes with the modular group, we have
\begin{align*}
\delta(\Delta_\phi^{it}\xi)=\partial(W(\Delta_\phi^{it}\xi,0))\Omega=\partial(\sigma^\psi_t(W(\xi,0)))\Omega=\sigma^\psi_t(\partial(W(\xi,0)))\Omega=\Delta_\psi^{it}\delta(\xi).
\end{align*}
Similarly,
\begin{align*}
\delta(J_\phi\xi)=\partial(W(J_\phi\xi,0))\Omega=\partial(\sigma^\psi_{-i/2}(W(\xi,0))^\ast)\Omega=\sigma^\psi_{-i/2}(\partial(W(\xi,0)))^\ast\Omega=J_\psi\delta(\xi).
\end{align*}
Thus $\delta$ is a symmetric derivation.

For $\xi\in D(T)^{\odot n}$ we have
\begin{align*}
\delta(\xi)&=\frac 1 i\lim_{t\searrow 0}\frac 1 t (\F(e^{it\Sigma})W(\xi,0)\F(e^{-it\Sigma})\Omega-W(\xi,0)\Omega)\\
&=\frac 1 i\lim_{t\searrow 0}\frac 1 t((e^{it\Sigma})^{\otimes n}(\xi,0)-(\xi,0))\\
&=\sum_{k=1}^n(\mathrm{id}^{k-1}\otimes \Sigma\otimes\mathrm{id}^{n-k})(\xi,0).
\end{align*}
Thus, if $\xi\in D(T)^{\odot m}$ and $\eta\in D(T)^{\odot n}$, then
\begin{align*}
\langle \delta(\xi),\delta(\eta)\rangle_\H&=\delta_{m,n}\sum_{k=1}^n \langle(\mathrm{id}^{k-1}\otimes \Sigma\otimes\mathrm{id}^{n-k})(\xi,0),(\mathrm{id}^{k-1}\otimes \Sigma\otimes\mathrm{id}^{n-k})(\eta,0)\rangle\\
&=n\delta_{m,n}\langle \xi,\eta\rangle\\
&=\E(\xi,\eta).
\end{align*}

\subsection{Compact quantum groups}\label{subsec:CQG}
Our last example are translation-invariant quantum Markov semigroups on compact quantum groups. We refer to \cite{Wor98} as a reference for the general theory of compact quantum groups and to \cite{CFK14} for the result on quantum Markov semigroups on compact quantum groups we need. Note that some of the common notation for compact quantum groups conflicts with the notation used in Tomita--Takesaki theory. For example, $S$ and $\phi$ are used for different objects than in the rest of the article.

A compact quantum group is a pair $(A,\Delta)$ consisting of a unital $C^\ast$-algebra $A$ and a unital $\ast$-homomorphism $\Delta\colon A\to A\otimes_{\min} A$ such that 
\begin{equation*}
(\Delta\otimes\id)\Delta=(\id\otimes\Delta)\Delta
\end{equation*}
and the sets
\begin{equation*}
\operatorname{lin}(A\otimes 1)\Delta(A),\,\operatorname{lin}(1\otimes A)\Delta(A)
\end{equation*}
are dense in $A\otimes_{\min}A$.

We will use the sumless Sweedler notation $\Delta(a)=a_{(1)}\otimes a_{(2)}$ for the comultiplication.

On a compact quantum group $A$ there exists a unique state $h$, called the Haar state, such that
\begin{equation*}
(h\otimes \mathrm{id})\Delta(a)=(\mathrm{id}\otimes h)\Delta(a)=h(a)1
\end{equation*}
for all $a\in A$.

Let $(\pi_h,L_2(A,h),\Omega_h)$ be the cyclic representation associated with $h$, and let $\M=\pi_h(A)^{\prime\prime}$. We also denote the vector state $\langle \Omega_h,\cdot\,\Omega_h\rangle$ on $\M$ by $h$.

An $n$-dimensional unitary corepresentation of $(A,\Delta)$ is a unitary $u\in M_n(A)$ such that
\begin{equation*}
\Delta(u_{jk})=\sum_{l=1}^n u_{jl}\otimes u_{lk}.
\end{equation*}
The linear span of the matrix coefficients $u_{jk}$, where $u$ runs over all unitary corepresentations of $(A,\Delta)$, forms a dense unital $\ast$-subalgebra $\A$ of $A$.

A unitary corepresentation $u$ is called irreducible if $\{u\}^\prime\cap M_n(\IC)=\IC 1$. Two unitary corepresentations $u$ and $v$ are called equivalent if there exists a unitary matrix $U\in M_n(\IC)$ such that $v=U^\ast u U$. If $(u^{(\alpha)})_{\alpha\in I}$ is a complete set of representatives of the equivalence classes of unitary corepresentations of $(A,\Delta)$ and $n_\alpha$ is the dimension of the corepresentation $u^{(\alpha)}$, then $\{u^{(\alpha)}_{jk}\mid \alpha\in I,\,1\leq j,k\leq n_\alpha\}$ is a basis of $\A$.

With the counit $\epsilon\colon \A\to\IC$ and the antipode $S\colon \A\to\A$ given by
\begin{align*}
\epsilon(u^{(\alpha)}_{jk})&=\delta_{jk}\\
S(u^{(\alpha)}_{jk})&=(u^{(\alpha)}_{kj})^\ast,
\end{align*}
the $\ast$-algebra $\A$ becomes a Hopf $\ast$-algebra.

For $f\colon \A\to\IC$ linear and $a\in \A$ the convolution is defined by $f\ast a=(\mathrm{id}\otimes f)\Delta(a)$, or, in Sweedler notation, $f\ast a=f(a_{(2)})a_{(1)}$. Convolution by the counit is the identity map. Similarly, if $f,g\in\A^\ast$, then $f\ast g=(f\otimes g)\Delta$.

Let $\phi\in \A^\ast$ be a hermitian functional such that $\phi(1)=0$, $\phi(a^\ast a)\geq 0$ if $a\in\ker\epsilon$ and $\phi\circ S=\phi$ and define the associated convolution semigroup of states by 
\begin{equation*}
\phi_t=\epsilon+\sum_{k=1}^\infty\frac{t^k}{k!}\phi^{\ast k}.
\end{equation*}
This induces to a semigroup $(P_t)$ on $\A$ given by $P_t=(\id\otimes\phi_t)\Delta$. The semigroup $(P_t)$ extends to a GNS-symmetric quantum Markov semigroup on $\M$, which is translation-invariant, that is, $\Delta P_t=(\mathrm{id}\otimes P_t)\Delta$ for all $t\geq 0$. Conversely, any translation-invariant GNS-symmetric quantum Markov semigroup on $\M$ is of this form \cite[Theorem 3.4, Corollary 4.6]{CFK14}. Let $\E$ denote the associated GNS-symmetric quantum Dirichlet form on $L_2(A,h)$.

The generating functional $\phi$ gives rise to a \emph{Schürmann triple} $((\pi,H),\eta,\phi)$, where $\pi$ is a contractive unital $\ast$-representation of $\A$ on the pre-Hilbert space $H$ and $\eta\colon \A\to H$ is a surjective linear map such that
\begin{align*}
\eta(ab)&=\pi(a)\eta(b)+\epsilon(b)\eta(a)\\
\langle \eta(a),\eta(b)\rangle&=\phi((a-\epsilon(a))^\ast (b-\epsilon(b)))
\end{align*}
for all $a,b\in\A$.

The bimodule and derivation associated with the quantum Dirichlet form $\E$ have been essentially (without the group $(\U_z)$ and the conjugation $\J$) described in \cite[Section 8]{CFK14}. Since $\A\Omega_h$ is a Tomita subalgebra of $\AA_h$ by \cite[Theorem 1.4]{Wor98} and core for $\E$ by \Cref{lem:approx}, we can restrict our attention to $\A\Omega_h$.

Let 
\begin{align*}
\H&=\operatorname{lin}\{a_{(1)}b\Omega_h\otimes \eta(a_{(2)})\mid a,b\in \A\},\\
L(a\Omega_h)(b\Omega_h\otimes\eta(c))&=ab\Omega_h\otimes\pi(a)\eta(c),\\
R(a\Omega_h)(b\Omega_h\otimes \eta(c))&=ba\Omega_h\otimes\eta(c),\\
\delta(a\Omega_h)&=a_{(1)}\Omega_h\otimes\eta(a_{(2)})
\end{align*}
for $a,b,c\in\A$. Note that $\H=\operatorname{lin}\{R(b\Omega_h)\delta(a\Omega_h)\mid a,b\in\A\}$.

It was shown in \cite[Proposition 8.1]{CFK14} that $\delta(\xi\cdot\eta)=L(\xi)\delta(\eta)+R(\eta)\delta(\xi)$ and $\E(\xi,\zeta)=\langle \delta(\xi),\delta(\zeta)\rangle_\H$ for $\xi,\zeta\in\A\Omega_h$. Clearly, $R(\xi)$ maps $\H$ into itself, and the product ensures that the same is the case for $L(\xi)$. In other words,
\begin{equation*}
L(a\Omega_h)(b_{(1)}c\Omega_h\otimes\eta(b_{(2)}))=(ab)_{(1)}c\Omega_h\otimes\eta((ab)_{(2)})-a_{(1)}bc\Omega_h\otimes\eta(a_{(2)}).
\end{equation*}

The group $(\U_z)$ can be described as follows. Let $\sigma$ be the modular group associated with $h$. Define
\begin{equation*}
\U_z(a_{(1)}b\Omega_h\otimes \eta(a_{(2)}))=\sigma_z(a)_{(1)}\sigma_z(b)\otimes \eta(\sigma_z(a)_{(2)}).
\end{equation*}
Clearly $\U_z\delta(a\Omega_h)=\delta(\sigma_z(a)\Omega_h)$ and $\U_z R(a\Omega_h)=R(\sigma_z(a)\Omega_h)\U_z$, which also implies $\U_z L(a\Omega_h)=L(\sigma_z(a)\Omega_h)\U_z$ by the product rule. Moreover,
\begin{align*}
&\quad\;\langle \U_z(a_{(1)}b\Omega_h\otimes \eta(a_{(2)})),c_{(1)}d\Omega_h\otimes \eta(c_{(2)})\rangle\\
&=h(\sigma_z(b)^\ast \sigma_z(a)^\ast_{(1)}c_{(1)}d)\phi((\sigma_z(a)_{(2)}-\epsilon(\sigma_z(a)_{(2)}))^\ast(c_{(2)}-\epsilon(c_{(2)})))\\
&=h(\sigma_z(b)^\ast \sigma_z(a)^\ast_{(1)}c_{(1)}d)(\phi(\sigma_z(a)_{(2)}^\ast c_{(2)})-\overline{\epsilon(\sigma_z(a)_{(2)})}\phi(c_{(2)})-\epsilon(c_{(2)})\overline{\phi(\sigma_z(a)_{(2)})}),
\end{align*}
where we used $\phi(1)=0$ in the last step. Now let us analyze the summands separately. For the first summand we have
\begin{align*}
h(\sigma_z(b)^\ast \sigma_z(a)^\ast_{(1)}c_{(1)}\phi(\sigma_z(a)_{(2)}^\ast c_{(2)}) d)&=h(\sigma_z(b)^\ast \phi\ast(\sigma_z(a)^\ast c)d)\\
&=h(\sigma_{-\bar z}(\sigma_{\bar z}(b^\ast)\phi\ast(\sigma_{\bar z}(a^\ast)c)d)\\
&=h(b^\ast \phi\ast(a^\ast \sigma_{-\bar z}(c))\sigma_{-\bar z}(d))
\end{align*}
since the modular group commutes with convolution by $\phi$ by GNS-symmetry [compare CFK].
For the second and third summand,
\begin{align*}
h(\sigma_z(b)^\ast \overline{\epsilon(\sigma_z(a)_{(2)})}\sigma_z(a)^\ast_{(1)}\phi(c_{(2)}) c_{(1)}d)&=h(\sigma_{\bar z}(b^\ast)(\epsilon\ast\sigma_{\bar z}(a^\ast))(\phi\ast c)d)\\
&=h(b^\ast a^\ast(\phi\ast \sigma_{-\bar z}(c))\sigma_{-\bar z}(d))
\end{align*}
and
\begin{align*}
h(\sigma_z(b)^\ast \overline{\phi(\sigma_z(a)_{(2)})}\sigma_z(a)^\ast_{(1)}\epsilon(c_{(2)})c_{(1)}d)&=h(\sigma_{\bar z}(b^\ast)(\phi\ast\sigma_{\bar z}(a^\ast))(\epsilon\ast c)d)\\
&=h(b^\ast (\phi\ast a^\ast)\sigma_{-\bar z}(c)\sigma_{-\bar z}(d))
\end{align*}
by the same reasoning.

Reversing these computations, we get
\begin{align*}
&\langle \U_z(a_{(1)}b\Omega_h\otimes \eta(a_{(2)})),c_{(1)}d\Omega_h\otimes \eta(c_{(2)})\rangle\\
&\qquad=\langle a_{(1)}b\Omega_h\otimes \eta(a_{(2)}),\U_{-\bar z}(c_{(1)}d\Omega_h\otimes \eta(c_{(2)}))\rangle.
\end{align*}
As for the conjugation operator $\J$, let
\begin{equation*}
\J(a_{(1)}b\Omega_h\otimes\eta(a_{(2)}))=\sigma_{i/2}(b)^\ast\sigma_{i/2}(a)^\ast_{(1)}\otimes \pi(\sigma_{i/2}(b)^\ast)\eta(\sigma_{i/2}(a)^\ast_{(2)}).
\end{equation*}
It is not hard to see that $\J R(a\Omega_h)=L(\sigma_{i/2}(a)^\ast\Omega_h)\J$, $\J\delta(a\Omega_h)=\delta(\sigma_{i/2}(a)^\ast\Omega_h)$ and $\U_z \J=\J\U_{\bar z}$.

Furthermore,
\begin{align*}
\J^2(a_{(1)}b\Omega_h\otimes\eta(a_{(2)}))&=\J(\sigma_{i/2}(b)^\ast\sigma_{i/2}(a)^\ast_{(1)}\otimes \pi(\sigma_{i/2}(b)^\ast)\eta(\sigma_{i/2}(a)^\ast_{(2)}))\\
&=\J(L(\sigma_{i/2}(b)^\ast \Omega_h)\delta(\sigma_{i/2}(a)^\ast\Omega_h))\\
&=\J(\delta(\sigma_{i/2}(ab)^\ast\Omega_h)-R(\sigma_{i/2}(a)^\ast\Omega_h)\delta(\sigma_{i/2}(b)^\ast\Omega_h))\\
&=\delta(ab\Omega_h)-L(a\Omega_h)\delta(b\Omega_h)\\
&=R(b\Omega_h)\delta(a\Omega_h)\\
&=a_{(1)}b\Omega_h\otimes\eta(a_{(2)}),
\end{align*}
so that $\J$ is an involution, and
\begin{align*}
&\langle \J(a_{(1)}b\Omega_h\otimes\eta(a_{(2)})),c_{(1)}d\Omega_h\otimes\eta(c_{(2)})\rangle\\
&\qquad=h(\sigma_{i/2}(a)\sigma_{i/2}(b)c_{(1)}d)\phi((\sigma_{i/2}(ab)_{(2)}-\epsilon(\sigma_{i/2}(a)_{(2)})\sigma_{i/2}(b))(c_{(2)}-\epsilon(c_{(2)}))).
\end{align*}
Considering all three non-zero summands separately as in the case of $\U_z$, one obtains 
\begin{align*}
&\langle \J(a_{(1)}b\Omega_h\otimes\eta(a_{(2)})),c_{(1)}d\Omega_h\otimes\eta(c_{(2)})\rangle\\
&\qquad=\langle a_{(1)}b\Omega_h\otimes\eta(a_{(2)}), \J(c_{(1)}d\Omega_h\otimes\eta(c_{(2)}))\rangle.
\end{align*}
Thus $\J$ and $(\U_z)$ make $\H$ into a Tomita bimodule over $\A\Omega_h$ and $\delta\colon\A\Omega_h\to\H$ is a symmetric derivation such that $D(\bar\delta)=D(\E)$ and $\langle \bar\delta(\xi),\bar\delta(\eta)\rangle=\E(\xi,\eta)$.

\DeclareFieldFormat[article]{citetitle}{#1}
\DeclareFieldFormat[article]{title}{#1} 

\printbibliography
\end{document}